\documentclass{amsart}
\usepackage{amssymb, amstext, amscd, amsmath}
\theoremstyle{plain}

\newtheorem{theorem}{Theorem}[section]
\newtheorem{lemma}[theorem]{Lemma}
\newtheorem{corollary}[theorem]{Corollary}
\newtheorem{proposition}[theorem]{Proposition}

\theoremstyle{definition}
\newtheorem{remark}[theorem]{Remark}
\newtheorem{example}[theorem]{Example}
\newtheorem{question}{Question}
\newcommand{\ad}{\textup{ad}}
\newcommand{\B}{\mathcal{B}}
\newcommand{\BH}{\mathcal{B(H)}}
\newcommand{\BPH}{\mathcal{B}_p\mathcal{(H)}}
\newcommand{\bC}{\mathbb{C}}

\newcommand{\cH}{{\mathcal{H}}}
\newcommand{\cK}{\mathcal{K}}
\newcommand{\cL}{\mathcal{L}}
\newcommand{\cI}{\mathcal{I}}
\newcommand{\CO}{\mathcal{O}}

\newcommand{\KH}{\mathcal{K(H)}}
\newcommand{\iso}{\textup{iso}}
\newcommand{\la}{\langle}
\newcommand{\ra}{\rangle}
\newcommand{\ran}{\textup{ran}}

\newcommand{\SC}{\mathcal{S}_C}
\newcommand{\SCP}{\mathcal{S}_{C,p}}
\newcommand{\tr}{\textup{tr}}

\begin{document}

\title[The orthogonal Lie algebra of operators]
{The orthogonal Lie algebra of operators: ideals and derivations}

\author[Q. Bu]{Qinggang Bu}
\address{Institute of Mathematics\\Jilin University\\Changchun 130012\\P. R. China}
\email{buqg18@mails.jlu.edu.cn}

\author[S. Zhu]{Sen Zhu}
\address{Department of Mathematics\\Jilin University\\Changchun 130012\\P. R. China}
\email{zhusen@jlu.edu.cn}

\date{\today}



\begin{abstract}
We study in this paper the infinite-dimensional orthogonal Lie algebra $\mathcal{O}_C$ which consists of all bounded linear operators
$T$ on a separable, infinite-dimensional, complex Hilbert space $\mathcal{H}$ satisfying $CTC=-T^*$, where $C$ is a conjugation on $\mathcal{H}$.
By employing results from the theory of complex symmetric operators and skew-symmetric operators, we determine the Lie ideals of $\mathcal{O}_C$ and their dual spaces. We study derivations of
$\mathcal{O}_C$ and determine their spectra. These results complete some results of P. de la Harpe and provide interesting contrasts between $\mathcal{O}_C$ and the algebra $\mathcal{B(H)}$ of
all bounded linear operators on $\mathcal{H}$.\end{abstract}

\maketitle
%%%%%%%%%%%%%%%%%%%%%%%%%%%5

 \section{Introduction}

The study of skew-symmetric matrices (i.e., those matrices $M$ satisfying $M+M^{tr}=0$, where $M^{tr}$ denotes the transpose of $M$) can be traced back to the work of L.-K. Hua on automorphic functions \cite{Hua,Hua1}, N. Jacobson on projective
geometry \cite{Jacobson}, and C. Siegel on symplectic
geometry \cite{Siegel}.
Skew-symmetric matrices arise naturally in partial differential equations \cite{Ralston}, differential geometry \cite{DG},  algebraic geometry \cite{Faenzi} and many other mathematical disciplines.
In the finite-dimensional Hilbert space of even dimension, there
is a natural one to one correspondence between the set of all non-singular
skew-symmetric matrices and the set of all sympletic forms on the Hilbert
space (\cite[page 16]{Kirillov}). As a classical finite-dimensional Lie algebra over the complex field $\bC$, the orthogonal Lie algebra $\textsl{so}(n,\bC)$
consists of all $n\times n$ skew-symmetric complex matrices, that is,
\[\textsl{so}(n,\bC)=\{X\in M_n(\bC): X+X^{tr}=0\},\]
where $M_n(\bC)$ denotes the set of all $n\times n$ complex matrices.
It is known that  $\textsl{so}(n,\bC)$ is the Lie algebra of the complex orthogonal Lie group
\[SO(n,\bC)=\{X\in M_n(\bC): XX^{tr}=I, \textup{det} X=1 \}\]
(see \cite[page 41]{Hall} or  \cite[page 341]{Helgason} for more details), and plays important roles in the classification of semi-simple Lie algebras.

Skew-symmetric matrices are playing important roles in quantum physics. In particular, they are closely related to the noncommutative spacetimes.
A noncommutative spacetime {\cite{03spacetime}} is defined
by the Hermitian generators $\hat x_i$ of a noncommutative $C^*$-algebra of ``functions on spacetime'' which obey the commutation relations
$$[\hat x_i, \hat x_j] = \textmd{i}\theta^{ij},$$
where $\theta^{ij}$ is a constant, real-valued skew-symmetric $d\times d$ matrix
($d$ is the dimension of spacetime).
There are many quantum problems based on the noncommutative space {\cite{gravity,dynamics,Renormalisation}}. In \cite{Kahng}, skew-symmetric matrices are used to construct $C^*$-algebraic locally compact quantum groups, which are deformation quantizations of Poisson-Lie groups induced by skew-symmetric matrices. In addition to the above, skew-symmetric matrices also appeared in noncommutative extensions of quantum mechanics
{\cite{weyl2,angular}}.

In the realm of applied mathematics, skew-symmetric matrices appear in the study of
coding theory \cite{Pinero}, cryptography \cite{CaoHu}, signal theory \cite{Agaian} and sampling theory \cite{Wisliceny}.
Especially many scientific and engineering applications give rise to eigenvalue problems for which the matrices are skew-symmetric.
So there are numerous papers devoted to computation of the eigenvalues of skew-symmetric matrices (see \cite{SIM-3,SIM-5}).

In view of the ubiquitous roles that skew-symmetric matrices play in various scientific branches,
naturally one may wish to develop a systematic theory of skew-symmetric matrices.
Skew-symmetric matrices as a whole have not received enough attention,
since only a small amount of work in the literature specialized in skew-symmetric matrices. The reader is referred to \cite{Gantmacher} for some basic properties of
skew-symmetric matrices. In \cite{Zagorodnyuk,Zagorodnyuk11,LiZhu}, an operator-theoretic approach to the structure of skew-symmetric matrices was developed.
To proceed, we introduce some notation and terminology.

Let $\cH$ be a complex, separable Hilbert space endowed with an inner product $\la\cdot,\cdot\ra$, and $\BH$ the algebra of all bounded linear operators on $\cH$. An operator $T\in \BH$ is said to be {\it skew-symmetric} if
$CTC=-T^*$ for some conjugation $C$ on $\cH$. Recall that a conjugate-linear map $C:\cH\rightarrow \cH$ is called a {\it conjugation} if $C$ is invertible with $C^{-1}=C$ and $\langle Cx, Cy\rangle=\langle y,x\rangle$ for all $x, y\in\cH$.

We remark that the term ``skew-symmetric" stems from the fact that each skew-symmetric operator on $\cH$ can be written as a skew-symmetric matrix (possibly infinite-dimensional) relative to an orthonormal basis of $\cH$.
In fact, given a conjugation $C$ on $\cH$, there exists an orthonormal basis~ $\{e_n\}_{n=1}^\infty$ of $\cH$ such that $Ce_n=e_n$ for all $n$ (see e.g. \cite[Lem. 2.11]{Gar14}).
Thus the matrix representation of an operator $X\in\BH$ relative to $\{e_n\}$ is $(a_{i,j})_{i,j\geq 1}$, where $a_{i,j}=\la Xe_j,e_i\ra$.
Then $CXC=-X^*$ if and only if
$$\la e_i,CXCe_j\ra=-\la e_i, X^*e_j\ra, \ \ \ \forall i,j\geq 1,$$ that is,
$$\la Xe_j,e_i\ra=-\la Xe_i,e_j\ra, \ \ \ \forall i,j\geq 1,$$ or equivalently $a_{i,j}=-a_{j,i}$ for all $i,j\geq 1$.
For the conjugation $C$, we denote
$$\CO_C=\{X\in\BH: CXC=-X^*\}.$$ Then, by the preceding discussion, $\CO_C$ is exactly the set of all operators $X$ on $\cH$ whose matrix representation with respect to $\{e_n\}$ is skew-symmetric.
Thus $\CO_C$ is exactly $\textsl{so}(n,\bC)$ when $\dim\cH=n$; if $\dim\cH=\infty$, then $\CO_C$ is an infinite analogue of $\textsl{so}(n,\bC)$.
This also shows that skew-symmetric operators is a generalization of skew-symmetric matrices in the setting of Hilbert space.

We remark that the term ``skew-symmetric operator" in literature is also used to denote linear operators $A$ (usually unbounded) densely defined on Hilbert spaces satisfying
\[\la Ax,y\ra=-\la x, Ay\ra\] for all $x,y$ in the domain of $A$ (see \cite{nelson,Muhly}). Obviously, this is quite different from the notion that we are discussing.

In the last decade, there has been growing interest in the study of skew-symmetric operators; see \cite{Ptak,Zagorodnyuk11,ZhuCAOT,2016oam,2017oam}.
In particular, skew-symmetric normal operators, partial isometries, compact operators and weighted shifts
are classified \cite{LiZhou,LiZhu,ZhuBJMA2015}.
Also we remark that the study of skew-symmetric operators is closely related to an important class of operators called complex symmetric operators (see Section \ref{S:LieIdeal} for the definition).
The reader is referred to \cite{LiZhu,2017oam} for more details.

The aim of the present paper is to study $\CO_C$ in the setting of Lie algebra for $C$ a conjugation on a separable, infinite-dimensional, complex Hilbert space $\cH$.
This is inspired by the observation that $\CO_C$ is a linear subspace of $\BH$ (closed in the weak operator topology), and is closed
under the Lie product
\[[X,Y]=XY-YX \ \ \ \textup{for all}\ \ X,Y\in\BH.\]  Thus, under the Lie product $[\cdot,\cdot]$, $\CO_C$ becomes a Lie algebra.
Hence, by the proceeding discussion, $\CO_C$ is an infinite-dimensional analogue of $\textsl{so}(n,\bC)$ in the setting of Hilbert space.
We call $\CO_C$ an infinite-dimensional orthogonal Lie algebra. It is known that $\CO_C$ is the Lie algebra of some Lie group (see \cite[Theorem 3]{Harpe}).

In \cite{Harpe}, P. de La Harpe discussed in detail many elementary aspects of $\CO_C$ (and several other classical Lie algebras of linear operators). Topics treated
include ideals, derivations, real forms, automorphisms and so on. Also, $\CO_C$ has been studied under the name of Cartan factor of type III for many years.
For example, $\CO_C$ is a concrete example of $J^*$-algebras. The latter was introduced and studied by L. Harris \cite{harris}
as a generalization of $C^*$-algebras. It was shown that basic theorems for $C^*$-algebras, such as functional calculus and the Kaplansky density theorem, can be generalized to $J^*$-algebras and hence to $\CO_C$.
In their paper \cite{Bunce} , L. Bunce, B. Feely and R. Timoney studied operator space structures of $\CO_C$ and some other type of Cartan factors as $JC^*$-triples.
An explicit construction of a universal ternary ring of operators (TRO) generated by $\CO_C$ was given. Also the universally reversibility
of $\CO_C$ was proved. All these work suggests a rich structure theory of $\CO_C$.

The authors' renewed interest in the orthogonal Lie algebra $\CO_C$ is connected with an effort to develop the theory of skew-symmetric operators.
This may provide a Lie algebraic approach to skew-symmetric operators. On the other hand, results concerning skew-symmetric
operators will in turn promote the study of orthogonal Lie algebras and their applications. In fact, one shall see later that the main results of this paper rely heavily on operator-theoretic arguments.

Given an algebraic object, the importance of the study towards its ideals is obvious.
The first aim of this paper is to describe the Lie ideal structure of $\CO_C$. A linear manifold $\mathcal{ L}$  in $\CO_C$ is called a Lie ideal of $\CO_C$ if $[A,X]\in  \mathcal{L}$ for every $A\in \CO_C$ and $X\in \mathcal{ L}$. In this paper linear manifolds are not assumed closed in  any topology. Note that the ideal structure of $\textsl{so}(n,\bC) (n\geq 1)$ has been clearly described (see e.g. \cite[Chapter 12]{Erdman}).  In fact, $\textsl{so}(n,\bC)$ has no nontrivial ideal unless $n=4$.

It was proved in \cite[page 78]{Harpe} that each nontrivial Lie ideal $\mathcal{L}$ of $\CO_C$ satisfies
$[\CO_C\cap\mathcal{F(H)}]\subset \mathcal{L}\subset [\CO_C\cap\KH]$ (see also \cite{Omori}), where $\mathcal{F(H)}$ is the set of all finite-rank operators on $\cH$ and $\KH$ is the set of all compact operators on $\cH$.  Thus, like $\BH$, $\CO_C$ will have many nontrivial ideals such as those induced by Schatten-$p$ classes.
Indeed, each associative ideal $\mathcal{I}$ of $\BH$ induces a Lie ideal of $\CO_C$, that is, $\CO_C\cap\mathcal{I}$.
By an associative ideal of $\BH$, we mean a two-sided ideal of $\BH$
under the usual multiplication of operators.
Thus it is natural to ask

\begin{question}\label{Q:main}
Is every Lie ideal of $\CO_C$ the intersection of $\CO_C$ and some associative ideal of $\BH$?
\end{question}

Note that any associative subalgebra of $\BH$ is a Lie algebra under the Lie product $[\cdot,\cdot]$.
In many cases, there is a close connection between the Lie ideal structure and the associative ideal structure of an associative subalgebra of $\BH$. This has been investigated by many people for associative subalgebras of $\BH$, such as $\BH$ \cite{Fong}, certain von Neumann algebras \cite{Meirs} and triangular operator algebras \cite{Hudson}.

The first result of this paper is the following theorem which classifies Lie ideals of $\CO_C$ and gives a positive answer to Question \ref{Q:main}.

\begin{theorem}\label{T:Lie}
Let $C$ be a conjugation on $\cH$ and $\mathcal{L}\subset\CO_C$. Then $\mathcal{L}$ is a Lie ideal of $\CO_C$ if and only if there exists an associative ideal $\cI$ of $\BH$ such that $\mathcal{L}=\cI\cap\CO_C$.
\end{theorem}

Although $\CO_C$ is not an associative algebra, the proceeding theorem shows that the Lie ideal structure is induced by the associative ideal structure of $\BH$. This reflects the close connection between $\CO_C$ and $\BH$, and hence suggests a rich structure of $\CO_C$.
In Subsection 2.1, we show that $\BH$ is Lie isomorphic to a Lie subalgebra of $\CO_C$. This means that $\BH$ is also ``contained" in $\CO_C$ and hence exhibits certain universality of $\CO_C$.
In view of these results, it is natural to ask whether some classical or older facts about $\BH$ still hold or have analogues in $\CO_C$.
Obtaining skew-symmetric analogues of some classical results about $\BH$ should provide an interesting contrast between $\CO_C$ and $\BH$. %This is one of our future goals.

In Subsection 2.2, we shall explore those norm ideals of $\CO_C$ induced by Schatten $p$-classes.

The Schatten $p$-class of compact operators on $\cH$ is denoted by $\BPH$, $1\leq p<\infty$.
It is well known that $\BPH$ is a Banach space under $p$-norm $\|\cdot\|_p$.
Moreover, $\BH$ is isometrically isomorphic to the  dual space of  $\B_1(\cH)$, and $\BPH$ is isometrically isomorphic to the  dual space of $\B_q(\cH)$ for $1< p, q<\infty$ with $\frac{1}{p}+\frac{1}{q}=1$.
For convenience, we denote $\B_\infty(\cH)=\KH$. So $\|\cdot\|_\infty$ means the operator norm. Then $\B_1(\cH)$ is isometrically isomorphic to the  dual space of  $\B_\infty(\cH)$.
The reader is refereed to \cite{Schatten} or \cite{Ringrose} for more details.

We shall explore Schatten $p$-classes in $\CO_C$.
For $p\in[1,\infty]$, we denote $\CO_{C,p}=\CO_C\cap\BPH$. Thus $\CO_{C,p}$ is closed in $\|\cdot\|_p$-norm, and is a Lie ideal of $\CO_C$. We establish in Subsection 2.2 the dual relations among $\CO_C$ and its Lie ideals $\CO_{C,p} (p\in[1,\infty])$ (see Proposition \ref{P:dual}).

The other aim of this section is to study derivations of $\CO_C$.
A linear map $\Phi:\CO_C\rightarrow\CO_C$ is called a {\it derivation} of $\CO_C$ if
\[\Phi([A, B])= [\Phi(A), B]+[A, \Phi (B)] \ \ \textup{~for all~} A, B\in\CO_C.\]

For $T\in\CO_C$, we can define a linear map $\ad_T$ on $\CO_C$ as $\ad_T(X)=[T, X]$ for all $X\in\CO_C$.
It is easy to check that $\ad_T$ is a derivation of $\CO_C$. On the other hand, by \cite[Chapter 2]{Harpe}, each derivation $\Phi$ of $\CO_C$ is inner, that is, $\Phi$ admits the form $\Phi=\ad_T$ for some $T\in\CO_C$.
Note that $\ad_T$ is a bounded linear operator on $\CO_C$. Thus one can see that $\ad: X\mapsto \ad_X$ is a
Lie algebra homomorphism from $\CO_C$ to $\B(\CO_C)$, the space of bounded linear operators on $\CO_C$.
The map $X\mapsto\ad_X$ is sometimes called the {\it adjoint map} or {\it adjoint representation}.

The notion of derivation plays important roles in the study of Lie algebras.
In particular, the Engel theorem asserts that a finite-dimensional Lie algebra is nilpotent
if and only if all its inner derivations are nilpotent. In infinite-dimensional case, many efforts are devoted to
the study of quasinilpotent Banach-Lie algebras (that is, all their inner derivations are quasinilpotent as bounded operators); see \cite{quasi3,quasi4,quasi1,quasi2}. %So we are interested in the spectral theory of derivations of $\CO_C$.

On the other hand, since some interesting results concerning the structure and the spectra theory of skew-symmetric operators have been obtained (e.g., \cite{LiZhu,ZhuCAOT,2016oam,ZhuZhao}),
it is natural and interesting to study the spectrum of $\ad_T$ in comparison with that of $T$ for $T\in\CO_C$. Also, note that each $\ad_T$ for $T\in\CO_C$ admits a continuous extension $\delta_T$ to $\BH$ defined by
$$\delta_T:X\longmapsto TX-XT.$$ The spectrum of $\delta_T$ and its different parts were determined by M. Rosenblum and D. Kleinecke
\cite{Lumer,Rosenblum}, C. Davis and P. Rosenthal \cite{DaviRosen}, and L. A. Fialkow \cite{Fialkow}. These motivate our present study towards the spectral theory of derivations of $\CO_C$.

In Section 3, we completely determine the spectrum of $\ad_T$ and its various parts for $T\in\CO_C$. This part depend on those results obtained in Subsection 2.2 concerning
the dual relations between Schatten $p$-classes in $\CO_C$. Our result shows that the spectrum of $\ad_T$ for $T\in\CO_C$ is obtained from the spectrum $\sigma(\delta_T)$ of $\delta_T$ by eliminating some of 
its isolated points (see Theorem \ref{T:SpecDeriva}).
Also we find for $T\in\CO_C$ that the left spectrum, the right spectrum, the approximate point spectrum and the approximate defect spectrum of $\ad_T$ all coincide. This phenomenon also happens to $\delta_T$ (see Remark \ref{R:comparison}).

For $T\in\CO_C$ and $p\in[1,\infty]$, note that $\CO_{C,p}$ is invariant under $\ad_T$.
Denote $\ad_{T,p}=\ad_T|_{\CO_{C,p}}$.  We view $\ad_{T,p}$ as a linear operator on $(\CO_{C,p}, \|\cdot\|_p)$. By {\cite[Thm. 2.3.10]{Ringrose}}, $\ad_{T,p}$ is bounded. It is natural to study the spectrum of $\ad_{T,p}$ in comparison with $\ad_T$.
Our result shows that $\sigma(\ad_{T,p})$ and its various parts coincide with that of $\sigma(\ad_{T})$ (see Theorem \ref{T:IdealSpecDeriva}).
As a corollary, we determine the spectra of derivations on $\textsl{so}(n,\bC)$.

\begin{theorem}\label{T:DerivFinDim}
Let $T\in \textsl{so}(n,\bC)$.  If $n\geq 2$, then $\sigma(\ad_T)=\{\lambda+\mu: \lambda,\mu\in\sigma(T)\}\setminus\{2z:z\in\Xi(T)\}$, where
\begin{align*}
  \Xi(T)=\{z\in\sigma(T): &~ \textup{rank}(T-z)^2=n-1 ~\textup{and there exist no } \\
  & \textup{distinct}~z_1,z_2\in \sigma(T) ~\textup{with}~ 2z=z_1+z_2\}.
\end{align*}
\end{theorem}

The preceding result indicates to certain extent that the operator-theoretic approach to skew-symmetric matrices is applicable. Indeed the reader will find that the proofs of main results rely heavily on some results from the theory of complex symmetric and skew-symmetric operators, which has received much attention since the seminal work of Garcia and Putinar \cite{Gar06}. Also, the authors believe that there will be concrete applications of the theory of skew-symmetric operators to other disciplines.

%, although this paper does not aim to provide some.

% the authors believe that there will be concrete applications of the theory of skew-symmetric operators to other disciplines,

%The preceding result indicates to certain extent that the operator-theoretic approach to skew-symmetric matrices is applicable.
%We remak that the approach to the present study is operator-theoretic.
% In fact, the proofs of main results rely heavily on some results
%from the theory of complex symmetric and skew-symmetric operators, which has received much attention since the seminal work of Garcia and Putinar \cite{Gar06}.
%The authors believe that there will be concrete applications of the theory of skew-symmetric matrices to other disciplines, although this paper does not aim to provide some.

In summary, results obtained in this paper complete some of P. de la Harpe's work \cite{Harpe}, and also provide some interesting contrasts between $\CO_C$ and $\BH$.
The authors are convinced that the Lie algebra $\CO_C$ deserves more study along this line.
It is our future goal to find more skew-symmetric analogues of some classical results about $\BH$.

The rest of this paper is organized as follows.
In Section 2, we shall determine ideals of $\CO_C$ and the dual spaces of $\CO_{C,p}$ for $p\in[1,\infty]$.
Section 3 is devoted to the characterization of the spectra of derivations of $\CO_C$ and $\CO_{C,p}$ for $p\in[1,\infty]$.

%%%%%%%%%%%%%%%%%%%%%%%%%%%%%%%%%%%%%%%%%%%%%%%%%%%%%%%%%%

\section{Lie ideals}\label{S:LieIdeal}

The aim of this section is to characterize the Lie ideals of $\CO_C$ for $C$ a conjugation on $\cH$.

The Lie ideal structures of Lie algebras are closely related to their classification theory, since there are so many notions in the theory of Lie algebras
connected to Lie ideals. For example, each finite-dimensional semisimple complex Lie algebra is the direct sum of some simple ideals \cite{Erdman}.
Theorem \ref{T:Lie} shows that the Lie ideal structure of $\CO_C$ depends on the associative ideal structure of $\BH$.
Also we shall prove later that $\BH$ is Lie isomorphic to a Lie subalgebra of $\CO_C$ (see Proposition \ref{P:universal}).
We feel that the results of this section may play some roles in the study toward Lie subalgebras of $\BH$.
%On the other hand, we shall prove later that $\BH$ is Lie isomorphic to a Lie subalgebra of $\CO_C$ (see Proposition \ref{P:universal}). So
%the classical theory of Operator Algebras may help to classify Lie subalgebras of $\BH$.
%On one hand,

%
%Given a Lie algebra, its ideals reflect much information concerning its structure. In fact,
%The results of this section have potential applications to the classification of Lie algebras of operators.

\subsection{Lie ideals of $\CO_C$}

The aim of this subsection is to prove Theorem \ref{T:Lie}.
We first make some preparation.

The following result can be verified directly.

\begin{lemma}\label{L:member}
Let $D$ be a conjugation on $\cH$ and $C=\begin{pmatrix}
0&D\\
D&0
\end{pmatrix}$. Then $C$ is a conjugation on $\cH\oplus\cH$ and \[\CO_C= \left\{ \begin{pmatrix}
A&E\\
F&-CA^*C
\end{pmatrix}:  A, E, F\in\BH ~\textup{with}~ E,F\in\CO_D\right\}.\]\end{lemma}

\begin{proposition}\label{P:CHARA}
Let $D$ be a conjugation on $\cH$ and $C=\begin{pmatrix}
0&D\\
D&0
\end{pmatrix}$. Assume that $\mathcal{L}$ is a Lie ideal of $\CO_C$.
Set \begin{align}\label{E0}
  \mathcal{L}_0=\left\{A\in\BH: \exists X_1,X_2,X_3\in\BH ~\textup{with}~  \begin{pmatrix}
A&X_1\\
X_2&X_3
\end{pmatrix}\in\mathcal{L}\right\},
\end{align}\[ \Delta_1=\vee\{XY+YDX^*D: X\in \mathcal{L}_0, Y\in\CO_D\}\] and
\[ \Delta_2=\vee\{DX^*D Y+YX: X\in \mathcal{L}_0,  Y\in\CO_D\},\]
where $\vee$ denotes linear span.
Then
\begin{enumerate}
\item[(i)] $\mathcal{L}_0$ is a Lie ideal of $\BH$ and \begin{equation}\label{E2}
\mathcal{L}=\left\{ \begin{pmatrix}
A&E\\
F&-DA^*D
\end{pmatrix} : A\in \mathcal{L}_0, E\in\Delta_1, F\in\Delta_2\right\};
\end{equation}
\item[(ii)] $XY\in \mathcal{L}_0$ for $X\in\Delta_1$ and $Y\in\CO_D$.
\end{enumerate}
\end{proposition}

\begin{proof} (i) For $X\in\BH$, we write $X^t=DX^*D$.

{\it Claim 1.} If $A\in \mathcal{L}_0$, then $A\oplus(-A^t)\in \mathcal{L}$.

In view of (\ref{E0}), there exists an element $T$ of $\mathcal{L}$ with form
\[T=\begin{pmatrix}
A&E\\
F&-A^t
\end{pmatrix},\] where $E,F\in\BH$ with $E=-E^t$ and $F=-F^t$.

Choose an invertible operator $G\in\CO_D$. %.$G\in\BH$  with $G=-G^t$.
Set \[  Y_1=\begin{pmatrix}
0&0\\
G&0
\end{pmatrix},\ \  Y_2=\begin{pmatrix}
0&G^{-1}\\
0&0
\end{pmatrix}.\] Thus, by Lemma \ref{L:member}, $Y_1,Y_2\in\CO_C$
and
\[Z_1:=-\frac{1}{2}[[T,Y_1],Y_1]=\begin{pmatrix}
0&0\\
GEG&0
\end{pmatrix}\in \mathcal{L}  \] and \begin{equation}\label{E4}
[Z_1,Y_2]=\begin{pmatrix}
-EG&0\\
0&GE
\end{pmatrix}\in \mathcal{L},\ \ \ -\frac{1}{2}[[Z_1,Y_2],Y_2]=\begin{pmatrix}
0&E\\
0&0
\end{pmatrix}\in \mathcal{L}.
\end{equation}
On the other hand,
\[Z_2:=-\frac{1}{2}[[T,Y_2],Y_2]=\begin{pmatrix}
0& G^{-1}FG^{-1}\\
0&0
\end{pmatrix}\in\mathcal{L}\] and
\begin{equation}\label{E5}[Z_2,Y_1]=\begin{pmatrix}
G^{-1}F&0\\
0&-F G^{-1}
\end{pmatrix}\in \mathcal{L},\ \ \ -\frac{1}{2}[[Z_2,Y_1],Y_1]=\begin{pmatrix}
0&0\\
F&0
\end{pmatrix}\in \mathcal{L}.
\end{equation}
So\begin{equation}\label{E6}
\begin{pmatrix}
A&0\\
0&-A^t
\end{pmatrix}=T-\begin{pmatrix}
0&E\\
0&0
\end{pmatrix}-\begin{pmatrix}
0&0\\
F&0
\end{pmatrix}\in \mathcal{L}.
\end{equation}

{\it Claim 2.} $\mathcal{L}_0$ is a Lie ideal of $\BH$.

Obviously, $\mathcal{L}_0$ is a linear manifold of $\BH$. Let $A\in\mathcal{L}_0$. Then, by Claim 1,
$$ \begin{pmatrix}
A&0\\
0&-A^t
\end{pmatrix}\in \mathcal{L}.$$For any $B\in\BH$, by Lemma \ref{L:member}, we have
$$ \begin{pmatrix}
B&0\\
0&-B^t
\end{pmatrix}\in \CO_C.$$ It follows that $$\begin{pmatrix}
[A,B]&0\\
0&[A^t,B^t]
\end{pmatrix}\in \mathcal{L}$$ and $[A,B]\in \mathcal{L}_0$. Thus this proves Claim 2.

%Also, by (\ref{E4}),(\ref{E5}) and (\ref{E6}), if $J $ contains an operator of form \[\begin{pmatrix}
%A&E\\
%F&-A^t
%\end{pmatrix},\] then $ A, EG, GF\in \mathcal{L}_0$ for any $G\in\BH$ with $G=-G^t$.

Denote by $\widetilde{\mathcal{L}}$ the set on the right side of (\ref{E2}). We shall prove that $\mathcal{L}=\widetilde{\mathcal{L}}$.

``$\mathcal{L}\subset \widetilde{\mathcal{L}}$". Choose an element $T\in \mathcal{L}$ of the form
\[T=\begin{pmatrix}
A&E\\
F&-A^t
\end{pmatrix},\] where $E=-E^t$ and $F=-F^t$. Also we fix an invertible $G\in\CO_D$. Then, from (\ref{E4}),(\ref{E5}) and (\ref{E6}), we obtain $A, EG,GF\in \cL_0$. So
it follows from Claim 1 that \begin{equation}\label{E1}
\begin{pmatrix}
A&0\\
0&-A^t
\end{pmatrix}\in \widetilde{\cL}.
\end{equation}  Also we note that
\[E=\frac{E+E}{2}=\frac{EGG^{-1}+G^{-1}GE}{2}=\frac{(EG)G^{-1}+G^{-1}(EG)^t}{2}.\]
Since $EG\in\cL_0$ and $G^{-1}\in\CO_D$, we obtain $E\in\Delta_1$ and
\begin{equation}\label{E7}\begin{pmatrix}
0&E\\
0&0
\end{pmatrix}\in\widetilde{\cL}.
\end{equation}
Likewise, noting that \[F=\frac{F +F}{2}=\frac{FG G^{-1}+G^{-1}GF}{2}=\frac{(GF)^tG^{-1}+G^{-1}(GF)}{2} \] and
$GF\in \cL_0$, it follows that  $F\in\Delta_2$ and\begin{equation}\label{E8} \begin{pmatrix}
0&0\\
F&0
\end{pmatrix}\in\widetilde{\cL}.
\end{equation}  Since $\widetilde{\cL}$ is obviously a linear manifold of $\B(\cH\oplus\cH)$, it follows from (\ref{E1}),(\ref{E7}) and (\ref{E8}) that
 $$T= \begin{pmatrix}
A&0\\
0&-A^t
\end{pmatrix}+\begin{pmatrix}
0&E\\
0&0
\end{pmatrix}+\begin{pmatrix}
0&0\\
F&0
\end{pmatrix}\in \widetilde{\cL}.$$

``$\widetilde{\cL}\subset \cL$". Assume that $$ T=  \begin{pmatrix}
A&EE_1+E_1E^t\\
F^t F_1+F_1F&-A^t
\end{pmatrix}, $$ where $A,E,F\in \cL_0$ and $E_1,F_1\in\CO_D$. By Claim 1, $\cL$ contains
$$   \begin{pmatrix}
A&0\\
0&-A^t
\end{pmatrix},\ \ R_1:=  \begin{pmatrix}
E&0 \\
0&-E^t
\end{pmatrix},\ \   R_2:=  \begin{pmatrix}
F&0\\
0&-F^t
\end{pmatrix}.$$
Note that $\CO_C$ contains
\[S_1:=\begin{pmatrix}
0&E_1\\
0&0
\end{pmatrix},\ \ S_2:=\begin{pmatrix}
0&0\\
F_1&0
\end{pmatrix}.\] Since $\cL$ is a Lie ideal of $\CO_C$, it follows that $\cL$ contains
\[[R_1,S_1]=\begin{pmatrix}
0&EE_1+E_1E^t\\
0&0
\end{pmatrix},\ \  [S_2,R_2]= \begin{pmatrix}
0&0\\
F^t F_1+F_1F&0
\end{pmatrix}.\]
This  implies
$$T= \begin{pmatrix}
A&0\\
0&-A^t
\end{pmatrix}+[R_1,S_1]+[S_2,R_2]\in \cL.$$

(ii) If $X\in\Delta_1$ and $Y\in\CO_D$, then, by (i) or by direct verification,
\[ W_1:=\begin{pmatrix}
0&X\\
0&0
\end{pmatrix}\in\cL, \ \ W_2:=\begin{pmatrix}
0&0\\
Y&0
\end{pmatrix}\in\CO_C.\]
Since $\cL$ is a Lie ideal of $\CO_C$, it follows that
\[ [W_1,W_2]=\begin{pmatrix}
XY&0\\
0&-YX
\end{pmatrix}\in\cL.\] In view of (\ref{E0}), we obtain $XY\in\cL_0$.
\end{proof}

Given two subsets $\mathcal{E}_1,\mathcal{E}_2$ of $\BH$, we let $[\mathcal{E}_1,\mathcal{E}_2]$ denote the linear span of the set
\[\{[A,B]: A\in \mathcal{E}_1, B\in\mathcal{E}_2\}.\]
A subset $\mathcal{E}$ of $\BH$ is said to be {\it invariant under similarity} if $STS^{-1}\in \mathcal{E}$ for all $T\in \mathcal{E}$
and all invertible $S\in\BH$.

\begin{lemma}[{\cite[Thm. 1 \& 2]{Fong}}]\label{L:LieIdealOfBH}
If $\cL$ is a linear manifold of $\BH$, then the following are equivalent:
\begin{enumerate}
\item[(i)]$\cL$ is a Lie ideal of $\BH$;
\item[(ii)]$\cL$ is invariant under unitary equivalence;
\item[(iii)]$\cL$ is invariant under similarity;
\item[(iv)] there exists an associative ideal $\cI$ of $\BH$ such that $[\cI, \BH]\subset \cL\subset\cI+\bC I$, where $I$ is the identity operator on $\cH$.
\end{enumerate}
\end{lemma}

Let $C$ be a conjugation on $\cH$. An operator $X\in\BH$ is said to be {\it $C$-symmetric} if $CXC=X^*$.
We denote by $\SC$ the set of all  $C$-symmetric operators on $\cH$. By \cite{Gar06}, each normal operator is $C$-symmetric for some conjugation $C$ on $\cH$.

\begin{lemma}\label{L:SimOrbCSO}
Let $C$ be a conjugation on $\cH$ and $\cI$ be an associative ideal of $\BH$. Then each self-adjoint operator in $\cI$ is unitarily equivalent to an operator in $\SC\cap\cI$.
\end{lemma}

\begin{proof}
Let $A$ be a self-adjoint operator in $\cI$. Then $A$ is complex symmetric and, by \cite[page 942]{2017oam}, there exists a unitary operator $U$ such that $UAU^*\in\SC$. By Lemma \ref{L:LieIdealOfBH}, $UAU^*\in\cI$. Thus $A$ is unitarily equivalent to an operator in $\SC\cap \mathcal{I}$.
\end{proof}

\begin{lemma}\label{L:invariant}
Let $C$ be a conjugation on $\cH$ and $\cI$ be an associative ideal of $\BH$. Then
$C \cI C=\cI $.
\end{lemma}

\begin{proof}
It suffices to prove $C \cI C\subset \cI $, since it immediately follows
$$\cI=(CC)\cI(CC)=C(C\cI C)C \subset C\cI C$$ and $C \cI C= \cI $.

``$C \cI C\subset \cI$". By \cite[Lemma I.5.1]{David}, $\cI$ is self-adjoint, that is, $X\in\cI$ implies $X^*\in \cI$. So it suffices to prove that $CAC\in\cI$ for any self-adjoint $A\in\cI$.

Now we fix a self-adjoint operator $A\in\cI$.
Clearly, $CAC$ is also self-adjoint. Then $CAC$ is complex symmetric and there exists a conjugation $D$ on $\cH$ such that $D(CAC)D=CAC$.

Set $U=CD$. Then $U\in\BH$ is unitary and $U^*AU=CAC$. Since $\cI$ is a Lie ideal of $\BH$, by Lemma \ref{L:LieIdealOfBH}, $\cI$ is invariant under unitary equivalence. So $CAC=U^*AU\in \cI$.
\end{proof}

Given a cardinality $n\geq 1$, we let $\cH^{(n)}$ denote the direct sum of $n$ copies of $\cH$.

\begin{lemma}\label{L:Combination}
Let $D$ be a conjugation on $\cH$ and $C=\begin{pmatrix}
0&D\\
D&0
\end{pmatrix}$. Let $\mathcal{I}$ be an associative ideal of $\B(\cH^{(2)})$ and
 $\cL$ be a linear manifold of $\B(\cH^{(2)})$ satisfying
$[\mathcal{I},\B(\cH^{(2)})]\subset\cL\subset\cI$.
Then
\begin{enumerate}
\item[(i)] $\vee\{XY+YCX^*C: X\in \cL, Y\in\CO_C \}=\CO_C\cap \cI$, and
\item[(ii)] if $ XY\in\cL$ for every $X\in\CO_C\cap \cI$ and $Y\in\CO_C$, then $\cL=\cI$.
\end{enumerate}
\end{lemma}

\begin{proof}
Denote $\Delta=\vee\{XY+YCX^*C: X\in \cL, Y\in\CO_C \}$.

(i) Since $\cI$ is an  associative ideal of $\B(\cH^{(2)})$ , there exists an  associative ideal $\cI_0$ of $\BH$ such that
$\cI=M_2(\cI_0).$

{\it Claim 1.} $\begin{pmatrix}
A&E\\
F&-A
\end{pmatrix}\in \cL$ for any $A,E,F\in\cI_0$.

Let $A,E,F\in\cI_0$. Since $$\begin{pmatrix}
0&-E\\
F&0
\end{pmatrix},\begin{pmatrix}
0&A\\
0&0
\end{pmatrix}\in M_2(\cI_0)=\cI,$$ it follows that
$$\frac{1}{2}\left[\begin{pmatrix}
0&-E\\
F&0
\end{pmatrix},\begin{pmatrix}
I&0\\
0&-I
\end{pmatrix}\right]=\begin{pmatrix}
0&E\\
F&0
\end{pmatrix}\in\cL $$ and
$$\left[\begin{pmatrix}
0&A\\
0&0
\end{pmatrix},\begin{pmatrix}
0&0\\
I&0
\end{pmatrix}\right]=\begin{pmatrix}
A&0\\
0&-A
\end{pmatrix}\in\cL.$$  This proves Claim 1.

{\it Claim 2.} $\Delta=\cI\cap\CO_C$.

 ``$\subset$". Fix $X\in \cL$ and $Y\in\CO_C$. It is easy to check that $XY+YCX^*C\in\CO_C$.
 On the other hand,  note that $X\in\cI$ and $\cI$ is self-adjoint. Then by Lemma \ref{L:invariant} $CX^*C\in\cI$.
So $XY+YCX^*C\in\CO_C\cap\cI$.

``$\supset$".
Arbitrarily choose $T\in \cI\cap\CO_C$. Then $T$ has the form of
\[T=\begin{pmatrix}
A&E\\
F&-DA^*D
\end{pmatrix},\]  where $E,F\in\cI_0\cap\CO_D$ and $A\in\cI_0$.

Since $E,F\in\cI_0$, by Claim 1, we have
$$\begin{pmatrix}
0&-E\\
F&0
\end{pmatrix}\in\cL.$$ On the other hand, it is easy to see that $$ \begin{pmatrix}
I&0\\
0&-I
\end{pmatrix}\in\CO_C.$$
Thus
\begin{align}\label{E12}
2\begin{pmatrix}
0&E\\
F&0
\end{pmatrix}&=\begin{pmatrix}
0&-E\\
F&0
\end{pmatrix}\begin{pmatrix}
I&0\\
0&-I
\end{pmatrix}+\begin{pmatrix}
I&0\\
0&-I
\end{pmatrix}\begin{pmatrix}
0&E\\
-F&0
\end{pmatrix}\\
&=\begin{pmatrix}
0&-E\\
F&0
\end{pmatrix}\begin{pmatrix}
I&0\\
0&-I
\end{pmatrix}+\begin{pmatrix}
I&0\\
0&-I
\end{pmatrix}C\begin{pmatrix}
0&-E\\
F&0
\end{pmatrix}^*C \in\Delta .
\end{align}

Now choose an invertible $G\in\CO_D$. It is easy to see that $$ \begin{pmatrix}
0&0\\
G^{-1}&0
\end{pmatrix}\in\CO_C.$$ For any $A\in\cI_0$, we have $AG\in\cI_0$ and, by Claim 1,
\[\begin{pmatrix}
0&AG\\
0&0
\end{pmatrix}\in \cL.\]
Thus
\begin{align*}\begin{pmatrix}
A&0\\
0&-DA^*D
\end{pmatrix}&=
\begin{pmatrix}
0&AG\\
0&0
\end{pmatrix}\begin{pmatrix}
0&0\\
G^{-1}&0
\end{pmatrix}-\begin{pmatrix}
0&0\\
G^{-1}&0
\end{pmatrix}\begin{pmatrix}
0&GDA^*D\\
0&0
\end{pmatrix}\\
&=\begin{pmatrix}
0&AG\\
0&0
\end{pmatrix}\begin{pmatrix}
0&0\\
G^{-1}&0
\end{pmatrix}+\begin{pmatrix}
0&0\\
G^{-1}&0
\end{pmatrix}C\begin{pmatrix}
0&AG\\
0&0
\end{pmatrix}^*C\in\Delta.
\end{align*}
Combing this with (\ref{E12}) produces that $T\in\Delta$. This proves Claim 2.

(ii) Now we assume that $ XY\in\cL$ for every $X\in\CO_C\cap \cI$ and $Y\in\CO_C$. We shall show that $\cL=\cI$.

{\it Claim 3.}
$\begin{pmatrix}
A&0\\
0&0
\end{pmatrix},\begin{pmatrix}
0&0\\
0&A
\end{pmatrix}\in \cL$ for all $A\in\cI_0\cap\mathcal{S}_D$.

Let $A\in\cI_0\cap\mathcal{S}_D$. Then $$\begin{pmatrix}
A & 0\\
0&-A
\end{pmatrix}=\begin{pmatrix}
A & 0\\
0&-DA^*D
\end{pmatrix}\in\CO_C\cap\cI.$$
Note that $$ \ \ \begin{pmatrix}
I & 0\\
0&-I
\end{pmatrix}\in\CO_C.$$ Then, by the hypothesis, we have
\[\begin{pmatrix}
A& 0\\
0&A
\end{pmatrix}=\begin{pmatrix}
A & 0\\
0&-A
\end{pmatrix} \begin{pmatrix}
I & 0\\
0&-I
\end{pmatrix}\in\cL.\]
By Claim 1, we have
\[\begin{pmatrix}
A & 0\\
0&-A
\end{pmatrix}\in\cL,\] which implies that $$\begin{pmatrix}
A & 0\\
0&0
\end{pmatrix},\begin{pmatrix}
0 & 0\\
0&A
\end{pmatrix}\in\cL.$$ This proves Claim 3.

Now we can conclude the proof.
By Lemma \ref{L:LieIdealOfBH}, $\cL$ is similarity-invariant. It follows from Lemma \ref{L:SimOrbCSO} and Claim 3 that
\[ \begin{pmatrix}
A & 0\\
0& B
\end{pmatrix}\in\cL\] for all self-adjoint operators $A,B\in\cI_0$.
Noting that $\cL$ is a linear manifold of $\B(\cH^{(2)})$,
it follows that
\[ \begin{pmatrix}
A & 0\\
0& B
\end{pmatrix}\in\cL, \ \ \forall A,B\in\cI_0.\]
This combing Claim 1 implies that $\cI=M_2(\cI_0)\subset \cL$. Furthermore, we obtain $\cI= \cL$.
\end{proof}

\begin{corollary}\label{C:Combination}
Let $\mathcal{I}$ be an associative ideal of $\BH$ and $\cL$ be a linear manifold of $\BH$ satisfying
$[\mathcal{I},\BH]\subset\cL\subset\cI$.
Assume that $C$ is a conjugation on $\cH$. Then
\begin{enumerate}
\item[(i)] $\vee\{XY+YCX^*C: X\in \cL, Y\in\CO_C \}=\CO_C\cap \cI$, and
\item[(ii)] if $ XY\in\cL$ for all $X\in\CO_C\cap \cI$ and all $Y\in\CO_C$, then $\cL=\cI$.
\end{enumerate}
\end{corollary}

\begin{proof}
Up to unitary equivalence, we may assume that $\cH=\cK^{(2)}$ for some Hilbert space $\cK$ and
\[C=\begin{pmatrix}
0&D\\
D&0
\end{pmatrix} \] for some conjugation $D$ on $\cK$.
In light of Lemma \ref{L:Combination}, one can see the conclusion.
\end{proof}

Now we are ready to prove Theorem \ref{T:Lie}.

\begin{proof}[Proof of Theorem \ref{T:Lie}]
The sufficiency is obvious. We need only prove the necessity.

``$\Longrightarrow$". We directly assume that $\cL$ is a nontrivial ideal of $\CO_C$.
Then, by \cite[page 78]{Harpe}, we have
\begin{align}\label{E90}
  [\CO_C\cap\mathcal{F(H)}]\subset\cL\subset[\CO_C\cap\KH].
\end{align}

Without loss of generality, we assume that $\cH=\cK^{(2)}$ for some Hilbert space $\cK$ and
\[C=\begin{pmatrix}
0&D\\
D&0
\end{pmatrix} \] for some conjugation $D$ on $\cK$. By Proposition \ref{P:CHARA},
there exists a Lie ideal $\cL_0$ of $\B(\cK)$ such that
\[\cL=\left\{ \begin{pmatrix}
A&E\\
F&-DA^*D
\end{pmatrix} : A\in \cL_0, E\in\Delta_1, F\in\Delta_2\right\},\]
where
\[ \Delta_1=\vee\{XY+YDX^*D: X\in \cL_0, Y\in\CO_D\}\] and
\[ \Delta_2=\vee\{DX^*D Y+YX: X\in \cL_0,  Y\in\CO_D\};\] moreover,
 $XY\in \cL_0$ for $X\in\Delta_1$ and $Y\in\CO_D$.

Since $\cL_0$ is a Lie ideal of $\B(\cK)$, it follows from Lemma \ref{L:LieIdealOfBH} that
we can find an associative ideal $\mathcal{I}_0$ of $\mathcal{B(K)}$ such that
$[\mathcal{I}_0,\mathcal{B(K)}]\subset\cL_0\subset  \cI_0+\bC I $, where $I$ is the identity operator on $\cK$.
In view of (\ref{E90}), each operator in $\cL$ is compact. So is $\cL_0$. Thus  $\cI_0 \subset\B_\infty(\cK)$ and $[\mathcal{I}_0,\mathcal{B(K)}]\subset\cL_0\subset  \cI_0$.

According to Corollary \ref{C:Combination}, we have
$\Delta_1=\CO_D\cap\cI_0$ and $\cL_0=\cI_0$. By Lemma \ref{L:invariant}, $\cI_0=D\cI_0D$. It follows that $\Delta_2=\Delta_1=\CO_D\cap\cI_0$.
Therefore
\[\cL=\left\{ \begin{pmatrix}
A&E\\
F&-DA^*D
\end{pmatrix} : A\in \cI_0, E,F\in \CO_D\cap\cI_0\right\}=M_2(\cI_0)\cap\CO_C.\]
Noting that $M_2(\cI_0)$ is an associative ideal of $\BH$, we conclude the proof.
\end{proof}

%
%The following proposition shows that $\CO_C$ admits some universality.

\begin{proposition}\label{P:universal}
Let $C$ be a conjugation on $\cH$. Then any Lie subalgebra of $\BH$ is Lie isomorphic to a Lie subalgebra of $\CO_C$.
\end{proposition}

\begin{proof}
Set $$D=\begin{pmatrix}
  0&C\\
  C&0
\end{pmatrix}.$$

For $X\in\BH$, define $\varphi:\BH\rightarrow \CO_D$ as
$$\varphi: X\longmapsto X\oplus(-CX^*C). $$
Then one can verify that $\varphi$ is a Lie homomorphism and
$\|\varphi(X)\|=\|X\|$ for all $X\in\BH$. Then any Lie subalgebra of $\BH$ is Lie isomorphic to a Lie subalgebra of $\CO_D$.
Note that $\CO_C\cong\CO_D$, that is, there exists unitary operator $U:\cH\rightarrow\cH\oplus\cH$ such that $U\CO_C U^*=\CO_D$. Hence the result follows readily.
\end{proof}

The result of Proposition \ref{P:universal} indicates that
the Lie algebra $\CO_C$ contains the information of all associative subalgebras of $\BH$.

%
%We denote by $\mathcal{F(H)}$ the collection of all finite-rank operators on $\cH$.
%
%\begin{corollary}
%Let $C$ be a conjugation on $\cH$. Then
%\begin{enumerate}
%\item[(i)] $\CO_C$ has a Lie subalgebra isomorphic to $\BH$,
%\item[(ii)]  $ \mathcal{F(H)}\cap\CO_C$ is the minimal  nontrivial Lie ideal of $\CO_C$, and
%\item[(iii)]  $ \mathcal{K(H)}\cap\CO_C$ is the maximal nontrivial Lie ideal of $\CO_C$.
%\end{enumerate}
%\end{corollary}

%
%\begin{theorem}\label{T:main}
%If $C$ is a conjugation on $\cH$, then $\CO_C\cap\KH$ is the unique nontrivial closed Lie ideal of $\CO_C$.
%\end{theorem}
%
%\begin{proof}
%Assume that $\cL$ is a nontrivial closed Lie ideal of $\CO_C$.  We shall show that $\CO_C\cap\KH=\cL$.
%
%``$\subset$". For $s,t\geq 1$, set $E_{s,t}=e_s\otimes e_t-e_t\otimes e_s$. By the hypothesis, $\cL$ contains a nonzero operator $A$. So there exists $i,j$ such that $\la Ae_i,e_j\ra\ne 0$.
%Then $[A,E_{i,i+j}]\ne 0$ is compact and $[A,E_{i,i+j}]\in\cL$. So one can deduce that $E_{s,t}\in\cL$ for all $s,t\geq 1$. So $\CO_C\cap\KH\subset\cL$.
%
%``$\supset$". For a proof by contradiction, we assume that $\cL\nsubseteq\cI$. So we can choose a non-compact $X\in\cL$.
%By Corollary \ref{C:nonCompact}, we obtain $\cL=\CO_C$, a contradiction.
%\end{proof}

\subsection{Schatten $p$-classes}
The aim of this subsection is to prove the following result, which classifies dual relations among $\CO_C$ and $\CO_{C,p} (p\in[1,\infty])$.

\begin{proposition}\label{P:dual}
Let $C$ be a conjugation on $\cH$. Then
\begin{enumerate}
\item[(i)] $(\CO_C,\|\cdot\|)$ is isometrically isomorphic to the  dual space of  $(\CO_{C,1}, \|\cdot\|_1)$;
\item[(ii)] $(\CO_{C,1}, \|\cdot\|_1)$ is isometrically isomorphic to the  dual space of  $(\CO_{C,\infty}, \|\cdot\|_\infty)$;
\item[(iii)] $(\CO_{C,p}, \|\cdot\|_p)$ is isometrically isomorphic to the  dual space of  $(\CO_{C,q}, \|\cdot\|_q)$ for $1< p, q<\infty$ with $\frac{1}{p}+\frac{1}{q}=1$.
\end{enumerate}
\end{proposition}

As an application of the preceding result, we shall describe in Subsection 3.2 the relations of derivations of $\CO_C$ and $\CO_{C,p}(p\in[1,\infty])$ (see Lemma \ref{L:Trace}).

In order to prove Proposition \ref{P:dual}, we need to make some preparation.

Given a complex matrix $A$, we denote by $A^{tr}$ the transpose of $A$.

\begin{lemma}\label{L:SchattenP}
Let $n$ be a positive integer.
If $A,B\in M_n(\bC)$ with $A=A^{tr}$ and $B=-B^{tr}$, then $\tr(AB)=0$, where $\tr(\cdot)$ is the trace function.
\end{lemma}

\begin{proof}
Note that $\tr(AB)=\tr(AB)^{tr}=\tr(B^{tr} A^{tr})=-\tr(B A)=-\tr(AB)$. So $\tr(AB)=0$.
\end{proof}

\begin{corollary}\label{C:SchattenP}
Let $C$ be a conjugation on $\cH$. Assume that $A\in\SC$ and $B\in\CO_C$.
If (i) $A\in \B_1(\cH)$, or (ii) $A\in\BPH$ and $B\in \B_q(\cH)$, where $1<p,q<\infty$ and $\frac{1}{p}+\frac{1}{q}=1$, then $\tr(AB)=0$.
\end{corollary}

\begin{proof}
We just give the proof in the case (i). The proof for the case (ii) is similar.

Since $C$ is a conjugation on $\cH$, by \cite[Lem. 2.11]{Gar14}, there exists an orthonormal basis $\{e_n\}$ such that $Ce_n=e_n$ for all $n$.
For each $n\geq 1$, denote by $P_n$ the projection of $\cH$ onto $\vee\{e_i: 1\leq i\leq n\}$.

Note that $P_n\rightarrow I$ in the strong operator topology. It follows that $\lim_n\|P_nAP_n-A\|_1=0$ and furthermore
$$\|P_nAP_nB-AB\|_1\leq \|P_nAP_n-A\|_1\cdot\|B\|\rightarrow 0$$ as $n\rightarrow \infty$. Thus $\tr(AB)=\lim_n\tr(P_nAP_nB)$.
It suffices to prove that $\tr(P_nAP_nB)=0$ for all $n$.

For each $n$, assume that
\[A=\begin{bmatrix}
A_n&*\\
*&*
\end{bmatrix}\begin{matrix}
\ran P_n \\ \ran (I-P_n)
\end{matrix},\ \ \ B=\begin{bmatrix}
B_n&*\\
*&*
\end{bmatrix}\begin{matrix}
\ran P_n \\ \ran (I-P_n)
\end{matrix}.\] It follows that $\tr(P_nAP_nB)=\tr(A_nB_n)$.
For $1\leq i,j\leq n$, note that
\begin{align*}
\la A_n e_i,e_j\ra&=\la A e_i,e_j\ra=\la A e_i, C e_j\ra\\
&=\la e_i, A^* C e_j\ra=\la e_i, C A e_j\ra\\
&=\la A e_j, C e_i \ra=\la A e_j, e_i \ra\\
&=\la A_n e_j, e_i \ra
\end{align*} and similarly that $\la B_n e_i,e_j\ra=-\la B_n e_i,e_j\ra$. Thus, relative to $\{e_1,e_2,\cdots,e_n\}$,
$A_n$ admits a symmetric matrix representation and $B_n$ admits a skew-symmetric matrix representation (that is, $R=-R^{tr}$).
By Lemma \ref{L:SchattenP}, $\tr(P_nAP_nB)=\tr(A_nB_n)=0$. Therefore we conclude that $\tr(AB)=0$.
\end{proof}

Given a Banach space $\mathcal X$, we let $\mathcal X'$ denote its dual space.

\begin{proof}[Proof of Proposition \ref{P:dual}]
(i) For $T\in \CO_C$, denote $$\psi_T:\CO_{C,1}\rightarrow\bC, \ \ \ \psi_T(X)=\tr(XT),\ \  \forall X\in\CO_{C,1}.$$
Then $|\psi_T(X)|\le \|X\|_1\cdot\|T\|$. Thus $\psi_T\in(\CO_{C,1})'.$

It suffices to prove that the map $\psi: T\rightarrow\psi_T$ is an isometric isomorphism of $\CO_C$ onto $(\CO_{C,1})'$. Clearly, $\psi$ is linear. It remains to check that $\psi$ is isometric and surjective.

{\it Step 1}. $\psi$ is isometric.

Fix a $T\in\CO_C$. Clearly, $\psi_K$ can be extended to the linear functional $\widetilde{\psi}_T$ on $\mathcal B_1(\cH)$ defined by \[\widetilde{\psi}_T(X)=\tr(XT), \ \ \forall X\in\mathcal B_1(\cH).\]Then, by \cite[Thm. 2.3.12]{Ringrose}, $\|\psi_T\|\le \|\widetilde{\psi}_K\|=\|T\|$ .

For any $X\in\mathcal B_1(\cH)$, denote $X_1=\frac{1}{2}(X+CX^*C)$ and $X_2=\frac{1}{2}(X-CX^*C)$. Note that $X_1\in\mathcal S_C\cap\mathcal B_1(\cH)$, $X_2\in\CO_{C,1}$ and $\|X_2\|_1\le\|X\|_1$. By Corollary 2.11,
\[|\widetilde{\psi}_T(X)|=|\tr(X_1T+X_2T)|=|\tr(X_2T)|=|\psi_T(X_2)|\le\|\psi_T\|\cdot\|X_2\|_1\le\|\psi_T\|\cdot\|X\|_1.\]Thus $\|\widetilde{\psi}_T\|\le\|\psi_T\|.$ Furthermore, we obtain $\|\psi_T\|=\|\widetilde{\psi}_T\|=\|T\|.$ It shows that $\psi$ is isometric.

{\it Step 2} $\psi$ is surjective.

For any $f\in (\CO_{C,1})'$, $f$ can be extended to a bounded linear functional $\tilde f$ on $\mathcal B_1(\cH)$, since $\CO_{C,1}$ is a subspace of $\mathcal B_1(\cH).$ Then, by \cite[Thm. 2.3.12]{Ringrose}, there exists an operator $T\in \mathcal B(\cH)$ such that $\tilde f(X)=\tr(XT)$ for all $X\in\mathcal B_1(\cH).$ Denote $T_1=\frac{1}{2}(T+CT^*C)$ and $T_2=\frac{1}{2}(T-CT^*C)$. Then $T_1\in\mathcal S_C$ and $T_2\in\CO_C.$ Thus, for each $X\in\CO_{C,1}$,
\[f(X)=\tilde f(X)=\tr(XT)=\tr(XT_1)+\tr(XT_2)=\tr(XT_2)=\psi_{T_2}(X).\] That is, $f=\psi_{T_2}.$ Hence $\psi$ is surjective.

(ii) Recall that $\CO_{C,\infty}=\KH\cap \CO_{C}$. For $T\in\CO_{C,1}$, denote $$\psi_T: \CO_{C,\infty}\rightarrow \mathbb C,\ \ \ \psi_T(X)=\tr(XT),\ \ \forall X\in\CO_{C,\infty}.$$
It is easily seen that $|\psi_T(X)|\le \|X\|\cdot\|T\|_1$. That is, $\psi_T\in(\CO_{C,\infty})'.$

In what follows, we will show that the map $\psi: T\rightarrow\psi_T$ is an isometric isomorphism of $\CO_{C,1}$ onto $(\CO_{C,\infty})'$. Clearly, $\psi$ is linear. It remains to check that $\psi$ is isometric and surjective.

{\it Step 1}. $\psi$ is isometric.

Fix a $T\in\CO_{C,1}$. It is easy to see that $\psi_T$ can be extended to the linear functional $\widetilde{\psi}_T$ on $\KH$ defined by \[\widetilde{\psi}_T(X)=\tr(XT),\ \ \forall X\in\KH.\]Then, by \cite[Thm. 2.3.12]{Ringrose}, $\|\psi_T\|\le \|\widetilde{\psi}_T\|=\|T\|_1$ .

For any $X\in\KH$, denote $X_1=\frac{1}{2}(X+CX^*C)$ and $X_2=\frac{1}{2}(X-CX^*C)$. Note that $X_1\in\mathcal S_C\cap\KH$, $X_2\in\CO_{C,\infty}$ and $\|X_2\|\le\|X\|$. By Corollary 2.11,
\[|\widetilde{\psi}_T(X)|=|\tr(X_1T+X_2T)|=|\tr(X_2T)|=|\psi_T(X_2)|\le\|\psi_T\|\cdot\|X_2\|\le\|\psi_T\|\cdot\|X\|.\]Hence $\|\widetilde{\psi}_T\|\le\|\psi_T\|.$ Furthermore, we obtain $\|\psi_T\|=\|\widetilde{\psi}_T\|=\|T\|_1.$ It shows that $\psi$ is isometric.

{\it Step 2}. $\psi$ is surjective.

For any $f\in(\CO_{C,\infty})'$, $f$ can be extended to a bounded linear functional $\tilde f$ on $\KH$, since $\CO_{C,\infty}$ is a subspace of $\KH.$ Then, by \cite[Thm. 2.3.12]{Ringrose}, there exists an operator $T\in \mathcal B_1(\cH)$ such that $\tilde f(X)=\tr(XT)$ for all $X\in\KH.$ Denote $T_1=\frac{1}{2}(T+CT^*C)$ and $T_2=\frac{1}{2}(T-CT^*C)$. Then $T_1\in\mathcal S_C\cap\KH$ and $T_2\in\CO_{C,\infty}.$ Thus, for each $X\in\CO_{C,\infty}$,
\[f(X)=\tilde f(X)=\tr(XT)=\tr(XT_1)+\tr(XT_2)=\tr(XT_2)=\psi_{T_2}(X).\] That is, $f=\psi_{T_2}.$ Hence $\psi$ is surjective.

(iii) The proof follows similar lines as those of (i) and (ii), and is omitted.
\end{proof}

%\begin{question}
%Let $\cL$ be a linear manifold of $\BH$ and $C$ be a conjugation on $\cH$.
%Denote $\cL^*:=\{A^*: A\in\cL\}$ and $C\cL C=\{CXC: X\in\cL \}$.
%
%1. If $\cL$ is a Lie ideal of $\BH$, then does it follow that $\cL^*$ and $C\cL C$ are Lie ideals of $\BH$?
%
%2. If $\cL$ is a Lie ideal of $\BH$, then does it follow that $\cL^*=\cL$ and $C\cL C=\cL$?
%
%3. If $\cL$ is a Lie ideal of $\BH$, then does it follow that $\cL\cap\CO=\vee\{EF+FE^t: E\in \cL, F\in\CO\}$?
%\end{question}

%\begin{question}
%1. Give a summary of results on the classification of finite dimensional Lie algebras.
%
%2. Describe representations of $\CO$.
%
%3. Lie Derivation?
%
%4. enveloping algebra of $\CO$? $\BH$?
%
%
%
%6. solvable
%
%7. simple, semi-simple, radical
%
%8. normalizer, Cartan subalgebra
%
%9. the derived algebra of $\KH$ is itself.
%\end{question}

%%%%%%%%%%%%%%%%%%%%%%%%%%%%%%%%%%%%%%%%%%
\section{Spectra of derivations}

%C. Davis and P. Rosenthal proved that $\sigma_\pi(\delta_A)=\sigma_l(A)-\sigma_r(A)$ and $\sigma_\delta(\delta_A)=\sigma_r(A)-\sigma_l(A)$.
%L. A. Fialkow proved that $\sigma_\pi(\delta_A)=\sigma_l(\delta_A)$ and $\sigma_\delta(\delta_A)=\sigma_r(\delta_A)$.
%
%
%C. Davis and P. Rosenthal, L. A. Fialkow
%
%L. A. Fialkow
%%%%%%%%%%%%%%%%%%%%%%%%%%%%%%%%%%%%%%%%%%
The aim of this section  is to describe the spectra of Lie derivations of $\CO_C$ for $C$ a conjugation on $\cH$.
To state our result, we introduce some terminology and notations.

Given a Banach space $\mathcal X$, we denote by $\mathcal{B(X)}$ the set of all bounded linear operators on $\mathcal X.$ For $A\in\mathcal{B(X)}$, we denote by $\ran A$ the range of $A$, and by $\ker A$ the kernel of $A$. We let $\sigma(A)$, $\sigma_l(A)$, $\sigma_r(A)$, $\sigma_\pi(A)$  and $ \sigma_\delta(T)$ denote respectively the spectrum, the left spectrum, the right spectrum, the approximate point spectrum and the approximate defect spectrum of $A$. That is,
\[\sigma_l(A)=\{z\in\bC : A-z \textup{ is not left invertible}\},\]
\[\sigma_r(A)=\{z\in\bC : A-z \textup{ is not right invertible}\},\]
\[\sigma_\pi(A)=\{z\in\bC : A-z \textup{ is not bounded below}\},\]
and \[\sigma_\delta(A)=\{z\in\bC : A-z \textup{ is not surjective}\}.\]
It is completely apparent that $\sigma_\pi(A)\subset\sigma_l(A)$ and $ \sigma_\delta(A)\subset\sigma_r(A)$.

Let $T\in\mathcal{B(H)}$  and suppose that $\lambda$ is an isolated point of $\sigma(T)$. Then
there exists an analytic Cauchy domain $\Omega$ such that
$\lambda\in\Omega$ and
$[\sigma(T)\setminus\{\lambda\}]\cap\overline{\Omega}=\emptyset$. We let
$E(\lambda; T)$ denote the {\it Riesz idempotent} of $T$
corresponding to $\lambda$, i.e.
\[E(\lambda;T)=\frac{1}{2\pi
\textup{i}}\int_{\Gamma}(z-T)^{-1}\textup{d}z,\] where
$\Gamma=\partial\Omega$ is positively oriented with respect to
$\Omega$ in the sense of complex variable theory. Denote $\mathcal{H}(\lambda;T)=\ran E(\lambda; T)$.
The reader is referred to \cite[Chapter 1]{Herr89} or \cite[Chapter VII]{ConwayFA} for more about Riesz idempotents.

Given a subset $\Gamma$ of $\bC$, we let $\iso\Gamma$ denote the set of all isolated points of $\Gamma$.

For $T\in\BH$, we denote by
$\Xi(T)$ the set of all isolated points $z$ of $\sigma(T)$ satisfying
\begin{enumerate}
\item[(a)] $\dim\cH(z;T)=1$,
\item[(b)] $z$ lies in the closure of the unbounded components of $\bC\setminus\sigma(T),$
\item[(c)] $2z\in\iso[\sigma(T)+\sigma(T)]$, and
\item[(d)] there exist no distinct $z_1,z_2\in\partial\sigma(T)$ such that $2z=z_1+z_2$.
\end{enumerate}
Obviously, $\Xi(T)$ is at most countable and $\Xi(2T)=\{2z:z\in\Xi(T)\}$.

The main result of this section is the following theorem.

\begin{theorem}\label{T:SpecDeriva}
Let $T\in\CO_C$. Then
\begin{enumerate}
\item[(i)] $\sigma(\ad_T)=\sigma_l(\ad_T)=\sigma_r(\ad_T)=\sigma_\pi(\ad_T)=\sigma_\delta(\ad_T)$;
\item[(ii)] $\sigma(\ad_T)=[\sigma(T)+\sigma(T)]\setminus\Xi(2T)$.
\end{enumerate}
\end{theorem}

For $T\in\CO_C$, Theorem \ref{T:SpecDeriva} implies that $\sigma(\ad_T)$ is obtained from $\sigma(T)+\sigma(T)$ by eliminating some isolated points.

%
%For $p\in \{0\}\cup [1, \infty)$, denote $\CO_{C,p}=\CO_C\cap\BP(\cH)$ and $\SCP=\SC\cap\BP(\cH)$. For $T\in\CO_C$, note that $\CO_{C,p}$ is invariant under $\ad_T$ and $\SCP$ is invariant under $U_T$.  Denote $\ad_{T,p}=(\ad_T)|_{\CO_{C,p}}$ and $U_{T,p}=(U_T)|_{\SCP}$ .  We view $\ad_{T,p}$ as a linear operator on $(\CO_{C,p}, \|\cdot\|_p)$, and  view $U_{T,p}$ as a linear operator on $(\SCP,\|\cdot\|_p)$. By {\cite[Thm. 2.3.10]{Ringrose}}, $\ad_{T,p}$ and $U_{T,p}$ are
%bounded.

%
%For $p\in \{0\}\cup [1, \infty)$, denote $\CO_{C,p}=\CO_C\cap\BP(\cH)$. For $T\in\CO_C$, note that $\CO_{C,p}$ is invariant under $\ad_T$.  Denote $\ad_{T,p}=(\ad_T)|_{\CO_{C,p}}$.  We view $\ad_{T,p}$ as a linear operator on $(\CO_{C,p}, \|\cdot\|_p)$. By {\cite[Thm. 2.3.10]{Ringrose}}, $\ad_{T,p}$ are bounded.

The following result describes the spectrum of $\ad_{T,p}$ for $T\in\CO_C$ and its different parts.

\begin{theorem}\label{T:IdealSpecDeriva}
Let $T\in\CO_C$ and $p\in[1,\infty]$. Then
$$\sigma(\ad_{T,p})=\sigma_l(\ad_{T,p})=\sigma_r(\ad_{T,p})=\sigma_\pi(\ad_{T,p})=\sigma_\delta(\ad_{T,p})=\sigma(\ad_T).$$
\end{theorem}

%%%%%%%%%%%%%%%%%%%%%%%%%%%%%%%%%%%%%%%%%%
\subsection{Approximate point spectra}
The aim of this subsection is to prove the following theorem, which describes approximate point spectra of
Lie derivations of $\CO_C$ and $\CO_{C,p}(p\in[1,\infty])$.

\begin{theorem}\label{T:app}
Let $T\in\CO_C$ and $p\in[1,\infty]$. Then
$$\sigma_\pi(\ad_{T})=\sigma_\pi(\ad_{T,p})=[\sigma(T)+\sigma(T)]\setminus\Xi(2T).$$

\end{theorem}

For the reader's convenience, we list some elementary facts, which will be used frequently.

For $e,f\in\cH$, we let $e\otimes f$ denote the operator on $\cH$ defined by $(e\otimes f) (x)=\la x,f\ra e$ for $x\in\cH$.

\begin{lemma}\label{L:basic}
Let $e,f\in\cH$ and $X=e\otimes f$. If $A\in\BH$ and $C$ is a conjugation on $\cH$, then
\begin{enumerate}
 \item[(i)] $AX=(Ae)\otimes f$,
 \item[(ii)] $XA=e\otimes(A^*f)$, and
 \item[(iii)] $CXC=(Ce)\otimes(Cf).$
\end{enumerate}
\end{lemma}

\begin{lemma}\label{L:NormEsit}
Let $e,f\in\cH$ with $\|e\|=\|f\|=1.$ Set $X=e\otimes (Cf)-f\otimes (Ce).$ Then
$X\in\CO_C$ and $1-|\la f,e\ra|^2\le\|X\|\le\|X\|_p\le 2$ for all $p\in [1,\infty].$
\end{lemma}

\begin{proof}
It is easy to check that $CXC=-X^*$ and $\|X\|\le\|X\|_p\le 2$.

On the other hand, compute to see
\begin{align*}
  \|X\| &\ge|\la XCf, e\ra|=|1-\la Cf, Ce\ra\la  f, e\ra|\\
  &= |1-\la e,f\ra\la f,e\ra|=1-|\la f,e\ra|^2.
\end{align*}
That is, $\|X\|\ge1-|\la f,e\ra|^2.$
\end{proof}

\begin{proposition}\label{P:key}
Let $T\in\CO_C$ and $p\in[1,\infty]$. If $\lambda,\mu\in\sigma_\pi(T)$ and $\lambda\ne\mu$, then
$\lambda+\mu \in[\sigma_\pi(\ad_{T,p})\cap\sigma_\pi(\ad_T)].$
\end{proposition}

\begin{proof}
Since $\lambda,\mu\in\sigma_\pi(T)$, there exist unit vectors $\{e_n, f_n: n\ge1\}$ such that $(T-\lambda)e_n\rightarrow 0$ and $(T-\mu)f_n\rightarrow 0$ as $n\rightarrow\infty.$ For $n\ge 1$, set $X_n=e_n\otimes (Cf_n)-f_n\otimes (Ce_n)$. By Lemma \ref{L:NormEsit}, $X_n\in\CO_C$ for $n\geq 1$.

{\it Claim.} $\liminf_{n}\|X_n\|> 0.$

Indeed, if not, then there exists a subsequence $\{X_{n_k}\}_{k\ge1}$ of $\{X_n\}_{n\ge 1}$ such that $\lim_{k}\|X_{n_k}\|=0$. By Lemma  \ref{L:NormEsit}, we get $$\|X_{n_k}\|\ge 1- |\la f_{n_k}, e_{n_k}\ra|^2\ge 0.$$Then $\lim_{k}|\la f_{n_k}, e_{n_k}\ra|=1.$  For $k\geq 1$, assume $f_{n_k}=\la f_{n_k}, e_{n_k}\ra e_{n_k}+g_{k}$ for some $g_{k}\in \{e_{n_k}\}^\bot$. Clearly, $g_k\rightarrow0$ as $k\rightarrow\infty$. It is easy to check that
\begin{align*}
(T-\lambda)f_{n_k}=\la f_{n_k}, e_{n_k}\ra(T-\lambda)e_{n_k}+(T-\lambda)g_k\rightarrow 0.
\end{align*}
Then, as $k\rightarrow\infty$, \[|\lambda-\mu|=\|(T-\lambda)f_{n_k}-(T-\mu)f_{n_k}\| \rightarrow 0.\]
This implies $\lambda=\mu$, a contradiction.

Therefore we have proved that $\liminf_{n}\|X_n\|> 0.$   Without loss of generality, we assume that $\inf_{n\geq 1}\|X_n\|> 0$. Thus
 \begin{align}\label{E20}
 2\geq \sup_{n\geq 1}\|X_n\|_p\geq\inf_{n\geq 1}\|X_n\|_p\geq\inf_{n\geq 1}\|X_n\|> 0.
 \end{align}

Compute to see
\begin{align*}
\ad_T(X_n)&=TX_n-X_nT\\
&=(Te_n)\otimes (Cf_n)-(Tf_n)\otimes(Ce_n)\notag\\
&\quad-e_n\otimes(T^*Cf_n)+f_n\otimes(T^*Ce_n)\\
&=(Te_n)\otimes (Cf_n)-(Tf_n)\otimes(Ce_n)\\
&\quad+e_n\otimes(CTf_n)-f_n\otimes(CTe_n).
\end{align*}
Note that
\begin{align*}
(\lambda+\mu)X_n&=(\lambda e_n)\otimes (Cf_n)-(\lambda f_n)\otimes (Ce_n)\\
&\quad+(\mu e_n)\otimes (Cf_n)-(\mu f_n)\otimes(Ce_n).
\end{align*}
%\[\lambda X_n=(\lambda e_n)\otimes f_n-(Cf_n)\otimes[C(\lambda e_n)],\]and
%\[\mu X_n= e_n\otimes (\bar{\mu}f_n)-[C(\bar{\mu}f_n)]\otimes (Ce_n).\]
Then, as $n\rightarrow \infty$,
\begin{align*}
\ad_T(X_n)-(\lambda+\mu)X_n&=[(T-\lambda)e_n]\otimes (Cf_n) -[(T-\mu)f_n]\otimes(Ce_n)\\
&\quad+e_n\otimes[C(T-\mu)f_n]-f_n\otimes[C(T-\lambda)e_n]\stackrel{\|\cdot\|_p}{\longrightarrow}0.
\end{align*}
In view of (\ref{E20}), we obtain $\lambda+\mu\in[\sigma_\pi(\ad_{T,p})\cap\sigma_\pi(\ad_T)].$
\end{proof}

\begin{lemma}\label{L:eigen}
Let $T\in\CO_C$ and $p\in[1,\infty]$.
 If $z\in\bC$ and $\dim\ker(T-z)^2\geq 2$, then
$2z\in[\sigma_\pi(\ad_{T,p})\cap\sigma_\pi(\ad_T)]$.
\end{lemma}

\begin{proof}
The proof is divided into two cases.

{\it Case 1.}  $\dim\ker(T-z)\geq 2$.

In this case, we can find unit vectors $e_1,e_2\in\ker(T-z)$ with $\la e_1,e_2\ra=0$.
Set $X=e_1\otimes(Ce_2)-e_2\otimes(Ce_1)$. Then $X$ is a nonzero finite-rank operator in $\CO_C$.
Compute to see
\begin{align*}
  \ad_T(X)&=TX-XT\\
  &=(Te_1)\otimes(Ce_2)-(Te_2)\otimes(Ce_1)\\
  &\quad -e_1\otimes(T^*Ce_2)+e_2\otimes(T^*Ce_1)\\
  &= ze_1 \otimes(Ce_2)-ze_2\otimes(Ce_1)\\
  &\quad +e_1\otimes(CTe_2)-e_2\otimes(CTe_1)\\
  &= ze_1 \otimes(Ce_2)-ze_2\otimes(Ce_1)\\
  &\quad +ze_1\otimes(Ce_2)-ze_2\otimes(Ce_1)\\
  &=2zX.
\end{align*} That is, $\ad_T(X)=2zX$. So
 $2z\in[\sigma_\pi(\ad_{T,p})\cap\sigma_\pi(\ad_T)]$.

{\it Case 2.}  $\dim\ker(T-z)=1$.

This implies that $\dim\ker(T-z)^2=2$. We can choose nonzero vectors $f_1,f_2\in\cH$ such that $(T-z)f_1=0$ and $(T-z)f_2=f_1$.

Set $X=f_1\otimes(Cf_2)-f_2\otimes(Cf_1)$. Then $X$ is a nonzero finite-rank operator in $\CO_C$.
Compute to see
\begin{align*}
TX&=(Tf_1)\otimes(Cf_2)-(Tf_2)\otimes(Cf_1)\\
&=zf_1\otimes(Cf_2)-(f_1+zf_2)\otimes(Cf_1)\\
&= z[f_1\otimes(Cf_2)- f_2 \otimes(Cf_1)]-f_1\otimes(Cf_1)\\
&= zX-f_1\otimes(Cf_1)
\end{align*} and
\begin{align*}
X T&= f_1\otimes(T^*Cf_2)-f_2\otimes(T^*Cf_1)\\
&= -f_1\otimes(CTf_2)+f_2\otimes(CTf_1)\\
&= -f_1\otimes[C(f_1+zf_2)]+f_2\otimes[C(zf_1)]\\
&=-z[f_1\otimes(C f_2)-f_2\otimes(Cf_1)]-f_1\otimes(C f_1)\\
&=-zX-f_1\otimes(C f_1).
\end{align*}
Thus $\ad_T(X)=TX-XT=2zX$. So
 $2z\in[\sigma_\pi(\ad_{T,p})\cap\sigma_\pi(\ad_T)]$.
\end{proof}

Let $\pi:\BH\rightarrow\BH/\KH$
be the natural map from $\BH$ into the Calkin algebra.
For $T\in\BH$, the left and right essential
spectrum of $T$ are defined by $\sigma_{le}(T)=\sigma_{l}(\pi(T))$ and $\sigma_{re}(T)=\sigma_{r}(\pi(T))$,
respectively.

\begin{lemma}\label{L:essSpect}
Let $T\in\CO_C$ and $p\in[1,\infty]$. If $z\in\iso\sigma(T)$ and $\dim\cH(z;T)=\infty$, then
$2z\in[\sigma_\pi(\ad_{T,p})\cap\sigma_\pi(\ad_T)]$.
\end{lemma}

\begin{proof}
By \cite[Prop. XI.6.9]{ConwayFA}, we have $z\in\sigma_{le}(T)\cap\sigma_{re}(T)$.
By \cite[Theorem 3.49]{Herr89}, there exists a compact $K\in\BH$ and an orthonormal sequence $\{f_n\}_{n\geq 1}$ such that $(T+K)f_n=zf_n$ for $n\geq 1$.
Then $(T-z)f_n\rightarrow 0$ as $n\rightarrow\infty$.
For $n\geq 1$, set $X_n=f_n\otimes( Cf_{n+1})-f_{n+1}\otimes(Cf_n)$.
Clearly, $X_n\in\CO_C$ is of finite rank and
$$2\geq \|X_n\|_p\geq \|X_n\|\geq \| X_n Cf_n\|=1.$$

Compte to see \begin{align*}
\ad_T(X_n)&=TX_n-X_nT\\
&=(Tf_n)\otimes (Cf_{n+1})-(Tf_{n+1})\otimes(Cf_n)\notag\\
&\qquad-f_n\otimes(T^*Cf_{n+1})+f_{n+1}\otimes(T^*Cf_n)\\
&=(Tf_n)\otimes (Cf_{n+1})-(Tf_{n+1})\otimes(Cf_n)\notag\\
&\qquad+f_n\otimes(C Tf_{n+1})-f_{n+1}\otimes(C Tf_n).
\end{align*}
Thus
\begin{align*}
\ad_T(X_n)-2zX_n&= [(T-z)f_n]\otimes (Cf_{n+1})-[(T-z)f_{n+1}]\otimes(Cf_n)\notag\\
&\qquad+f_n\otimes[C (T-z)f_{n+1}]-f_{n+1}\otimes[C (T-z)f_n].
\end{align*} tends to $0$ as $n\rightarrow\infty$. So $2z\in[\sigma_\pi(\ad_{T,p})\cap\sigma_\pi(\ad_T)]$.
\end{proof}
%
%For $T\in\BH$, we denote by
%$\Xi_0(T)$ the set of all isolated points $z$ of $\sigma(T)$ satisfying
%\begin{enumerate}
%\item[(i)] $\dim\cH(z;T)=1$, and
%\item[(ii)] $z$ lies in the closure of the unbounded components of $\bC\setminus\sigma(T).$
%\end{enumerate}

To discuss points in $\Xi(T)$, we need a key lemma.

\begin{lemma}\label{L:plane}
If $\Omega$ is a nonempty, bounded, connected, open subset of $\bC$, then
$(\Omega+\Omega)\subset(\partial\Omega+\partial\Omega)$.
\end{lemma}

\begin{proof}
Arbitrarily choose $z_1,z_2\in \Omega$. We shall prove there exist $z_3,z_4\in\partial\Omega$ such that $z_1+z_2=z_3+z_4$.

Without loss of generality, we assume that $\textup{dist}(z_1,\partial\Omega)\leq\textup{dist}(z_2,\partial\Omega)$. Then we can find $\alpha\in\bC$ such that $z_1+\alpha\in\partial\Omega$ and $z_2-\alpha\in\overline{\Omega}$. Clearly, if $z_2-\alpha\in\partial\Omega$, then we are done. So in the sequel we assume $z_2-\alpha\in\Omega$. Denote $w_1=z_1+\alpha$ and $w_2=z_2-\alpha$.

For $z\in\overline{\Omega}$, we define $h(z)=w_2+w_1-z$. Clearly, $h(\partial\Omega)\cap\Omega\ne\emptyset$, since $h(w_1)=w_2\in\Omega$.

{\it Claim.}  $h(\partial\Omega)\cap(\bC\setminus\Omega)\ne\emptyset$.

Indeed, if not, then $h(\partial\Omega)\subset\Omega$. Since $h(\partial\Omega)$ is a nonempty compact set, it follows that
the diameter of $h(\partial\Omega)$ is less than that of $\Omega$. This is a contradiction, since $h$ is a rigid transformation and the diameter of $h(\partial\Omega)$ coincides with that of $\partial\Omega$ and $\Omega$. This proves the claim.

Since $\Omega$ is connected, it follows from the proceeding claim that $h(\partial\Omega)\cap\partial\Omega\ne\emptyset$.
So there exists $z_3\in\partial\Omega$ such that $h(z_3)\in\partial\Omega$. Denote $z_4=h(z_3)$. Then $$z_3+z_4=z_3+h(z_3)=w_2+w_1=z_1+z_2,$$
which completes the proof.
\end{proof}

The following two corollaries are clear.

\begin{corollary}\label{C:plane1}
If $\Gamma$ is a nonempty, compact subset of $\bC$ and $z$ is an interior point of $\Gamma$, then there exist
distinct $z_1,z_2\in\partial\Gamma$ such that $z_1+z_2=2z$.
\end{corollary}

\begin{corollary}\label{C:plane}
If $\Gamma$ is a nonempty, compact subset of $\bC$, then
$\Gamma+\Gamma=\partial\Gamma+\partial\Gamma$.
\end{corollary}

\begin{lemma}\label{L:SSOBasic}
Let $T\in\CO_C$. Then
\begin{enumerate}
  \item[(i)] $\sigma(T)=-\sigma(T)$,
  \item[(ii)] $\sigma_l(T)=\sigma_\pi(T)=-\sigma_\delta(T)=-\sigma_r(T)$,
 \item[(iii)] $\Xi(T)=-\Xi(T)$.
\end{enumerate}
\end{lemma}

\begin{proof}
(i) For any $\lambda\in\bC$, note that
\begin{align}\label{E030}
  C(T-\lambda)C=-(T+\lambda)^*.
\end{align} It is easy to see that $T-\lambda$ is invertible if and only if so is $(T+\lambda)^*$. Thus $\lambda\in\sigma(T)$ if and only if $-\lambda\in\sigma(T)$. So $\sigma(T)=-\sigma(T)$.

(ii) Since $T$ acts on a Hilbert space, it is well known that $\sigma_l(T)=\sigma_\pi(T)$ and $\sigma_\delta(T)=\sigma_r(T)$.
On the other hand, one can see from (\ref{E030}) that $T-\lambda$ is bounded blow if and only if so is $(T+\lambda)^*$, which equals that $T+\lambda$ is surjective. This proves (ii).

(iii) Let $z\in\Xi(T)$. Since $\sigma(T)=-\sigma(T)$, it follows that $-z\in\iso\sigma(T)$, $-z$ lies in the unbounded components of $[\bC\setminus\sigma(T)]\cup\{-z\}$ and there exist no distinct $z_1,z_2\in\partial\sigma(T)$ such that $2z=z_1+z_2$.
On the other hand, by \cite[Thm. 1.7 (i)]{ZhuZhao}, we have $\dim\cH(-z;T)=1$. Thus $-z\in \Xi(T)$.
\end{proof}

\begin{corollary}\label{C:sum}
If $T\in\CO_C$, then
$$\partial\sigma(T)+\partial\sigma(T)=\sigma_\pi(T)+\sigma_\pi(T)=\sigma(T)+\sigma(T)=-[\sigma_\pi(T)+\sigma_\pi(T)].$$
\end{corollary}
%
%To prove the proceeding lemma, we need an auxiliary result, which is an easy observation. For the reader's convenience, we write down its proof.

\begin{proof}
By \cite[Prop. VII.6.7]{ConwayFA}, we have $\partial\sigma(T)\subset\sigma_\pi(T)$.
Then, by Corollary \ref{C:plane}, $$ \sigma(T)+ \sigma(T)= [\partial\sigma(T)+\partial\sigma(T)]\subset [ \sigma_\pi(T)+ \sigma_\pi(T)]\subset [ \sigma(T)+ \sigma(T)].$$
So $\partial\sigma(T)+\partial\sigma(T)=\sigma_\pi(T)+\sigma_\pi(T)=\sigma(T)+\sigma(T)$. Since $\sigma(T)=-\sigma(T)$, it follows that
$$\sigma(T)+\sigma(T)=-[\sigma(T)+\sigma(T)]=-[\sigma_\pi(T)+\sigma_\pi(T)].$$
This completes the proof.
\end{proof}

Let $\mathcal{X}$ be a Banach space and $A\in\mathcal{B(X)}$.
Recall that $A$ is called a {\it semi-Fredholm} operator, if
$\ran A$ is closed and either $\dim\ker A$ or $\dim\mathcal{X}/\ran A$ is finite; in this case,
$\textup{ind} A:=\dim\ker A-\dim  \mathcal{X}/\ran A$ is called the {\it index} of $A$.

\begin{proposition}\label{L:key}
Let $T\in\CO_C$ and $p\in[1,\infty]$.
 If $z\in[\sigma_\pi(T)\setminus\Xi(T)]$, then
$2z\in[\sigma_\pi(\ad_{T,p})\cap\sigma_\pi(\ad_T)]$.
\end{proposition}

\begin{proof}
If $z$ is an accumulation point of $\sigma_\pi(T)$, then we can find $\{z_n: n\geq 1\}\subset\sigma_\pi(T)\setminus\{z\}$ such that $z_n\rightarrow z$. So $z_n+z\rightarrow 2z$.
By Proposition \ref{P:key}, $z_n+z\in[\sigma_\pi(\ad_{T,p})\cap\sigma_\pi(\ad_T)]$ for  $n\geq 1$. It follows that $2z\in[\sigma_\pi(\ad_{T,p})\cap\sigma_\pi(\ad_T)]$.

In the following we assume that $z\in\iso\sigma_\pi(T)$. Then there exists $r>0$ such that $T-w$ is bounded below for $w\in B(z,r)\setminus \{z\}$. Moreover, there exists nonnegative integer $n$ or $n=\infty$ such that
$\textup{ind}(T-w)=-n$ for all $w\in B(z,r)\setminus \{z\}$.

Now the proof is divided into four cases.

{\it Case 1.}  $n>0$.

This means that $B(z,r)\subset\sigma(T)$. So $z$ lies in the interior of $\sigma(T)$.
By Corollary \ref{C:plane1}, there exist distinct $z_1,z_2\in\partial\sigma(T)$ such that $z_1+z_2=2z$.
Note that $\partial\sigma(T)\subset\sigma_\pi(T)$. By Proposition \ref{P:key}, we have $2z\in[\sigma_\pi(\ad_{T,p})\cap\sigma_\pi(\ad_T)]$.

{\it Case 2.} $n=0$.

Since $T-w$ is bounded below for $w\in B(z,r)\setminus \{z\}$, this means that $T-w$ is invertible. So $z\in\iso\sigma(T)$.

{\it Case 2.1.} $\dim\cH(z;T)=\infty$.

In view of Lemma \ref{L:essSpect}, this implies $2z\in[\sigma_\pi(\ad_{T,p})\cap\sigma_\pi(\ad_T)]$.

{\it Case 2.2.} $1<\dim\cH(z;T)<\infty$.

This means that $\dim\ker(T-z)^2\geq 2$.
By Lemma \ref{L:eigen}, we obtain $2z\in[\sigma_\pi(\ad_{T,p})\cap\sigma_\pi(\ad_T)]$.

{\it Case 2.3.} $ \dim\cH(z;T)=1$.

Note that $z\notin\Xi(T)$. Then there are three possibilities:
(i) $z$ lies in a bounded component $\Omega$ of $[\bC\setminus\sigma(T)]\cup\{z\}$;
(ii) $2z$ is an accumulation point of $\sigma(T)+\sigma(T)$;
(iii) There exist distinct $z_1,z_2\in\partial\Omega$ such that $2z=z_1+z_2$.

By Proposition \ref{P:key}, we can see in case (iii) that $2z\in[\sigma_\pi(\ad_{T,p})\cap\sigma_\pi(\ad_T)]$.

In case (i), by Corollary \ref{C:plane1}, there exist distinct $w_1,w_2\in\partial\Omega$ such that $2z=w_1+w_2$.
Note that $\partial\Omega\subset\partial\sigma(T)\subset\sigma_\pi(T)$.
Using Proposition \ref{P:key} again, we have $2z\in[\sigma_\pi(\ad_{T,p})\cap\sigma_\pi(\ad_T)]$.

Now we consider case (ii). By Corollary \ref{C:sum}, we can choose $\{\lambda_n,\mu_n: n\geq 1\}\subset\partial\sigma(T)$ such that $\lambda_n+\mu_n\ne 2z$ for all $n$ and $\lambda_n+\mu_n\rightarrow 2z$.
Without loss of generality, we may directly assume that $\lambda_n\rightarrow\lambda_0$ and
$\mu_n\rightarrow\mu_0$. Thus $\lambda_0,\mu_0\in\partial\sigma(T)$ and $\lambda_0+\mu_0=2z$.
If $\lambda_0\ne\mu_0$, then, by Proposition \ref{P:key}, we have $2z\in[\sigma_\pi(\ad_{T,p})\cap\sigma_\pi(\ad_T)]$.
If $\lambda_0=\mu_0$, then $\lambda_0=\mu_0=z$.
Since $\lambda_n+\mu_n\ne 2z$ for all $n$, either $\{\lambda_n\}$  or $\{\mu_n\}$ has a subsequence, denoted as $\{w_{k}\}$,
converging to $z$ and $w_k\ne z$ for all $k$. Then, by Proposition \ref{P:key}, $w_{k}+z\in[\sigma_\pi(\ad_{T,p})\cap\sigma_\pi(\ad_T)]$.
Since $w_{k}+z\rightarrow 2z$, it follows that $2z\in[\sigma_\pi(\ad_{T,p})\cap\sigma_\pi(\ad_T)]$.
Hence we conclude the proof.
\end{proof}

\begin{lemma}\label{L:Defect}
Let $T\in \CO_C$. Then
\begin{enumerate}
    \item[(i)] $\sigma_\pi(\ad_T)\subset [\sigma(T)+\sigma(T)],$
    \item[(ii)] $\sigma_\delta(\ad_T)\subset[\sigma(T)+\sigma(T)],$
    and
    \item[(iii)] $\sigma(\ad_T) \subset [\sigma(T)+\sigma(T)].$
  \end{enumerate}
\end{lemma}

\begin{proof}
(i)
Note that $\BH=\CO_C+\SC$ and $\CO_C\cap\SC=\{0\}$, that is, $\SC$ is a topology complement of $\CO_C$. For $X\in\SC$, define $U_T(X)=[T,X]$. One can check that $U_T$ is a bounded linear operator on $\SC$.  Thus
  \begin{align}\label{E9}
    \delta_{T}=\begin{pmatrix}
                \ad_T&0 \\ 0&U_T
  \end{pmatrix}
  \begin{matrix}
    \CO_C\\
  \SC  ,
  \end{matrix}
  \end{align} which implies that $\sigma_\pi(\ad_{T})\subset \sigma_\pi(\delta_{T})$.
In view of Lemma \ref{L:SSOBasic} (ii) and {\cite[Thm. 3.19]{Herr89}}, we have $\sigma_\pi(\delta_{T})=\sigma_\pi(T)+\sigma_\pi(T)\subset[\sigma(T)+\sigma(T)]$. Thus (i) follows readily.

(ii)  In view of (\ref{E9}), we have $\sigma_\delta(\ad_{T})\subset\sigma_\delta(\delta_{T})$. Using \cite[Thm. 3.19]{Herr89}, we deduce that $\sigma_\delta(\delta_{T})=\sigma_\delta(T)-\sigma_\pi(T)$. By Lemma \ref{L:SSOBasic} (i) and Corollary \ref{C:sum}, we deduce that $\sigma_\delta(\delta_{T})=\sigma({T})+\sigma({T})$. Thus (ii) follows readily.

(iii) Since $\sigma(\ad_{T})=\sigma_\pi(T)\cup\sigma_\delta(T)$, the result is clear.
\end{proof}

For $p\in[1,\infty]$, denote $\SCP=\SC\cap\BPH$. Then $\SCP$ is a subspace of $\BPH$, which is closed in $\|\cdot\|_p$ norm.

\begin{lemma}\label{L:quasi}
Let $T\in \CO_C$ and $p\in [1, \infty]$. Then
\begin{enumerate}
    \item[(i)] $\sigma_\pi(\ad_{T,p})\subset [\sigma(T)+\sigma(T)],$
    \item[(ii)] $\sigma_\delta(\ad_{T,p})\subset[\sigma(T)+\sigma(T)].$
  \end{enumerate}
\end{lemma}

\begin{proof}
Note that $\BPH=\CO_{C,p}+\SCP$ and $\CO_{C,p}\cap\SCP=\{0\}$, that is, $\SCP$ is a topology complement of $\CO_{C,p}$.
  Thus
  \begin{align}\label{E22}
    \delta_{T}|_{\BPH}=\begin{pmatrix}
                \ad_{T,p}&0 \\ 0&U_{T,p}
  \end{pmatrix}
  \begin{matrix}
   \CO_{C,p}\\ \SCP
  \end{matrix},
  \end{align} where $U_{T,p}=U_{T}|_{\SCP}$.
It is clear that
$$\sigma_\pi(\ad_{T,p})\subset \sigma_\pi(\delta_{T}|_{\BPH}),\ \ \sigma_\delta(\ad_{T,P})\subset\sigma_\delta(\delta_{T}|_{\BPH}). $$

In view of Theorems 3.53 and 3.54 in {\cite{Herr89}}, we have
$\sigma_\pi(\delta_{T}|_{\BPH})=\sigma_\pi(T)-\sigma_\delta(T)$ and
$\sigma_\delta(\delta_{T}|_{\BPH})=\sigma_\delta(T)-\sigma_\pi(T)$.
Then the desired result follows readily from Lemma \ref{L:SSOBasic} and Corollary \ref{C:sum}.
\end{proof}

\begin{lemma}\label{L:thrBythr}
Let $C$ be a conjugation on $\cH$ with form
\[C=\begin{pmatrix}
0&0&C_3\\
0&C_2&0\\
C_1&0&0
\end{pmatrix}\begin{matrix}
\cH_1\\ \cH_2 \\ \cH_3
\end{matrix},\] where $\cH_1\oplus \cH_2\oplus\cH_3=\cH$ and $\dim\cH_1=\dim\cH_3=1$.
Then $C_2$ is a conjugation on $\cH_2$, $C_3=C_1^{-1}$ and \[D:=\begin{pmatrix}
 0&C_3\\
C_1&0
\end{pmatrix}\begin{matrix}
\cH_1 \\ \cH_3
\end{matrix} \] is a conjugation on $\cH_1\oplus\cH_3$; moreover,
an operator $R$ lies in $\CO_C$ if and only if $R$ can be written as
\begin{align}\label{E-temp}
  R=\begin{pmatrix}
Y&Z&0\\
W&X&-C_2Z^*C_3\\
0&-C_1W^*C_2&-C_1Y^*C_3
\end{pmatrix}\begin{matrix}
\cH_1\\ \cH_2 \\ \cH_3
\end{matrix},
\end{align} where $X\in \CO_{C_2}, Y\in \B(\cH_1), Z\in \B(\cH_2,\cH_1)$ and $W\in \B(\cH_1,\cH_2).$
\end{lemma}

\begin{proof}
The first half of the result is an easy verification.
Now we turn to the latter half.

Let $R\in\BH$. Assume that \begin{align*}
  R=\begin{pmatrix}
Y&Z&V\\
W&X& Z_1\\
V_1&W_1&Y_1
\end{pmatrix}\begin{matrix}
\cH_1\\ \cH_2 \\ \cH_3
\end{matrix}.
\end{align*}
One can see from direct calculation that $R\in\CO_C$ if and only if
\begin{enumerate}
\item[(i)] $X\in \CO_{C_2}$,
\item[(ii)] $Y_1=-C_1Y^*C_3$, $Z_1=-C_2Z^*C_3$, $W_1=-C_1W^*C_2$, and
\item[(iii)] $ \begin{pmatrix}
0&V\\
0&0
\end{pmatrix}\begin{matrix}
\cH_1\\ \cH_3
\end{matrix}, \ \ \begin{pmatrix}
0&0\\
V_1&0
\end{pmatrix}\begin{matrix}
\cH_1\\ \cH_3
\end{matrix}\in \CO_D.$
\end{enumerate}
Note that $\dim \cH_1\oplus\cH_3=2$. By \cite[Cor. 3.3]{ZhuZhao}, statement (iii) holds if and only if $V=0$ and $V_1=0$.
This completes the proof.
\end{proof}

\begin{lemma}\label{L:eliminate}
Let $T\in\CO_C$ and $p\in[1,\infty]$. If $z\in\Xi(T)$, then $2z\notin[\sigma_\pi(\ad_T)\cup\sigma_\pi(\ad_{T,p})]$.
\end{lemma}

\begin{proof}
First we claim that $0\notin\Xi(T)$. In fact, if $\sigma(T)=\{0\}$, then $\cH(0;T)=\cH$ is of infinite dimension; if
$\sigma(T)$ contains nonzero elements, then we can choose a nonzero $w\in\partial\sigma(T)$. Since $\sigma(T)=-\sigma(T)$, it follows that $-w\in\partial\sigma(T)$ and $0=(-w)+w$. So we have proved in either case that $0\notin\Xi(T)$.  Hence $z\ne 0$.

Since $z\in\Xi(T)\setminus\{0\}$, it follows that $\{z\}$ and $\{-z\}$ are two disjoint clopen subsets of $\sigma(T)$. Denote $\Gamma=\sigma(T)\setminus\{\pm z\}$. Then $\Gamma$ is a clopen subset of $\sigma(T)$.
Set $\cH_1=\cH(z;T)$, $\cK=\cH(\Gamma\cup\{z\};T)$ and $\cH_3=\cH\ominus \cK$. Then $\cH_1,\cK$ are both invariant under $T$ and $\cH_1\subset \cK$. Denote $\cH_2=\cK\ominus\cH_1$. Then $T$ can be written as
\begin{align}\label{E03}
  T=\begin{pmatrix}
A_1&E&G\\
0&A&F\\
0&0&A_3
\end{pmatrix}\begin{matrix}
\cH_1\\ \cH_2 \\ \cH_3
\end{matrix}.
\end{align}

By Riesz's decomposition theorem, $\sigma(A_1)=\{z\}$, $\sigma(A_3)=\{-z\}$ and $\sigma(A)=\Gamma$.
Note that $z,-z\in\Xi(T)$. Thus $\dim\cH_1=1=\dim\cH_3$. So $A_1=z I_1$ and $A_3=-zI_3$, where
$I_i$ is the identity operator on $\cH_i$, $i=1,3$.

Assume that $\cH_1=\vee\{e\}$, where  $e$ is a unit vector. Note that $(T+z)^*Ce=-C(T-z)e=0$. Thus $Ce\in\ker(T+z)^*$.
From the matrix representation (\ref{E03}), it can be seen that $\ker(T+z)^*=\cH_3$. So $\cH_3=\vee\{Ce\}$, $C(\cH_1)=\cH_3$
and equivalently $C(\cH_3)=\cH_1$. Since $C$ is isometric and $C=C^{-1}$, one can see that $C(\cH_2)=\cH_2$. Thus $C$ can be written as
\[C=\begin{pmatrix}
0&0&C_3\\
0&C_2&0\\
C_1&0&0
\end{pmatrix}\begin{matrix}
\cH_1\\ \cH_2 \\ \cH_3
\end{matrix}. \] Clearly, $C_3=C_1^{-1}$ and $C_2$ is a conjugation on $\cH_2$. Noting that $CTC=-T^*$, we have $C_2AC_2=-A^*$, that is, $A\in\CO_{C_2}$.

%For $X\in \B(\cH_2)$, denote $\dot{X}=-C_2X^*C_2$.

For $Y\in \B(\cH_1)$, denote $\tilde{Y}=-C_1Y^*C_3$. For $Z\in \B(\cH_2,\cH_1)$\ denote $ \hat{Z}=-C_2Z^*C_3$.  For $W\in \B(\cH_1,\cH_2)$, denote $\check{W}=-C_1W^*C_2$. Then, by Lemma \ref{L:thrBythr}, an operator $R$ lies in $\CO_C$ if and only if $R$ can be written as
\begin{align}\label{E-Repren}
  R=\begin{pmatrix}
Y&Z&0\\
W&X&\hat{Z}\\
0&\check{W}&\tilde{Y}
\end{pmatrix}\begin{matrix}
\cH_1\\ \cH_2 \\ \cH_3
\end{matrix},
\end{align} where   $X\in \CO_{C_2}, Y\in \B(\cH_1), Z\in \B(\cH_2,\cH_1)$ and $W\in \B(\cH_1,\cH_2).$
Compute to see
\begin{align}\label{E001}
(\ad_T-2z)(R)&=(T-z)R-R(T+z)\\
&=\begin{pmatrix}
EW-2zY&EX-YE-Z(A+z)&*\\
(A-3z)W&(\ad_A-2z)X+\hat{E}\check{W}-WE&*\\
*&*&*
\end{pmatrix}.
\end{align}

{\it Claim 1.}  $3z\notin\sigma(A)$.

For a proof by contradiction, we assume that $3z\in\sigma(A)$. So $3z\in\sigma(T)$. Note that  $-z\in\sigma(T)$. If $3z$ is an interior point of $\sigma(T)$, then
$2z=3z+(-z)$ is an interior point of $\sigma(T)+\sigma(T)$, contradicting the hypothesis that $2z\in\iso[\sigma(T)+\sigma(T)]$.
Thus $3z\in\partial\sigma(T)$. So $2z$ is the sum of two distinct points  $3z$, $-z$ in $\partial\sigma(T)$, contradicting the fact that $z\in\Xi(T)$.

{\it Claim 2.}  $2z\notin[\sigma(A)+\sigma(A)]$.

For a proof by contradiction, we assume that $2z\in[\sigma(A)+\sigma(A)]$.
By Corollary \ref{C:plane}, we can choose $z_1,z_2\in \partial\sigma(A)$ such that
$2z=z_1+z_2$. Since $z\notin\sigma(A)$, it follows that $z_1\ne z_2$. Note that $\partial\sigma(A)\subset \partial\sigma(T)$.  Thus $2z$ is the sum of two distinct points in $\partial\sigma(T)$, contradicting the fact that $z\in\Xi(T)$.

By Lemmas \ref{L:Defect} and \ref{L:quasi}, Claim 2 implies that both $2z-\ad_A$ and $2z-\ad_{A,p}$ are invertible for $p\in[1,\infty]$.

Now it suffices to show $2z\notin[\sigma_\pi(\ad_T)\cup\sigma_\pi(\ad_{T,p})]$. Assume that $\{R_n:n\geq 1\}\subset\CO_C$ and $(\ad_T-2z)(R_n)\rightarrow 0$ in the norm topology.
In view of (\ref{E-Repren}), we may assume that
\begin{align*}
  R_n=\begin{pmatrix}
Y_n&Z_n&0\\
W_n&X_n&\hat{Z}_n\\
0&\check{W}_n&\tilde{Y}_n
\end{pmatrix}\begin{matrix}
\cH_1\\ \cH_2 \\ \cH_3
\end{matrix},
\end{align*} where   $X_n\in \CO_{C_2}, Y_n\in \B(\cH_1), Z_n\in \B(\cH_2,\cH_1)$ and $W_n\in \B(\cH_1,\cH_2).$

Since $\|(\ad_T-2z)(R_n)\|\rightarrow 0$, it follows from (\ref{E001}) that
\begin{align}\label{E-Reprenn}&\|EW_n-2zY_n\| +\|EX_n-Y_nE-Z_n(A+z)\|\\
&+\|(A-3z)W_n\|+\|(\ad_A-2z)X_n+\hat{E}\check{W_n}-W_nE\| \longrightarrow 0.\end{align}
So Claim 1 implies that $\|W_n\|\rightarrow 0$, and hence $\|\check{W_n}\|\rightarrow 0$ . In view of (\ref{E-Reprenn}), this means $\|Y_n\|\rightarrow 0$ and $\|(\ad_A-2z)X_n\|\rightarrow 0 $.
In view of Lemma \ref{L:Defect} and Claim 2, we have $\|X_n\|\rightarrow 0$. By (\ref{E-Reprenn}), we obtain $\|Z_n(A+z)\|\rightarrow 0$.
Recall that $-z\notin\sigma(A)$. So $\|Z_n\|\rightarrow 0$ and hence $\|\hat{Z_n}\|\rightarrow 0.$
Therefore we conclude that $\|R_n\|\rightarrow 0$. This shows that  $2z\notin \sigma_\pi(\ad_T)$.

Using a similar argument as above, one can show that  $2z\notin \sigma_\pi(\ad_{T,p})$.
\end{proof}

Now we are ready to prove Theorem \ref{T:app}.

\begin{proof}[Proof of Theorem \ref{T:app}]
By Lemma \ref{L:eliminate}, Propositions \ref{P:key} and \ref{L:key}, we have
\[[\sigma_\pi(T)+\sigma_\pi(T)]\setminus\Xi(2T)\subset \sigma_\pi(\ad_T).\]
By Lemmas \ref{L:Defect} and \ref{L:eliminate}, we have
\[\sigma_\pi(\ad_T)\subset [\sigma(T)+\sigma(T)]\setminus\Xi(2T).\]
In view of Corollary \ref{C:sum}, we obtain
\begin{align}\label{E-ap}
\sigma_\pi(\ad_T)=[\sigma(T)+\sigma(T)]\setminus\Xi(2T).
\end{align}

For $p\in [1,\infty]$, by Propositions \ref{P:key} and \ref{L:key}, we have
\[[\sigma_\pi(T)+\sigma_\pi(T)]\setminus\Xi(2T)\subset \sigma_\pi(\ad_{T,p}).\]
By Lemmas \ref{L:quasi} and \ref{L:eliminate}, we have
\[\sigma_\pi(\ad_{T,p})\subset [\sigma(T)+\sigma(T)]\setminus\Xi(2T).\]
In view of Corollary \ref{C:sum}, we obtain
\begin{align}\label{E-app}
\sigma_\pi(\ad_{T,p})=[\sigma(T)+\sigma(T)]\setminus\Xi(2T).
\end{align}
\end{proof}
%%%%%%%%%%%%%%%%%%%%%%%%%%%%%%%
%\subsection{Adjoints and approximate defect spectra}
%This subsection is devoted to estimating approximate defect spectra of Lie derivations of $\CO_C$ .

\subsection{Proofs of Theorems \ref{T:SpecDeriva} and \ref{T:IdealSpecDeriva}}
As an application of the result in Subsection 2.2, we shall describe relations of derivations of $\CO_C$ and $\CO_{C,p}(p\in[1,\infty])$.
Then we shall give the proofs of Theorems \ref{T:SpecDeriva} and \ref{T:IdealSpecDeriva}.

Given a Banach space $\mathcal{X}$, we let $\mathcal{X}'$ denote its dual space. If $A: \mathcal{X}\rightarrow \mathcal{X}$ is a bounded linear operator, we denote by $A'$ the adjoint of $A$ acting on $\mathcal{X}'$.
Two operators $A,B$ acting respectively on two Banach spaces $\mathcal{X}_1$ and $\mathcal{X}_2$ are said to be {\it similar}, denoted as $A\sim B$, if there exists an invertible bounded linear operator $S:\mathcal{X}_1\rightarrow \mathcal{X}_2$ such that $SA=BS$.
%
%The main result of this subsection is the following proposition, which describes the connections among $\ad_T$ and $\ad_{T,p} (p\in[1,\infty])$.

\begin{proposition}\label{L:Trace}
If $T\in\CO_C,$ then
\begin{enumerate}
\item[(i)] $\ad_T\sim -\ad_{T,1}'$,
\item[(ii)] $\ad_{T,1}\sim -\ad_{T,\infty}'$, and
\item[(iii)] $\ad_{T,p}\sim -\ad_{T,q}'$ for $p,q\in(1,\infty)$ with $\frac{1}{p}+\frac{1}{q}=1$.
\end{enumerate}
\end{proposition}

\begin{proof}
We just prove statement (i), since the proofs of (ii) and (iii) follow similar lines.

(i)   By Proposition \ref{P:dual}, there exists the isometrical isomorphism $\psi$ of $\CO_C$ onto $(\CO_{C,1})'$, which is defined by $\psi(Z)=\psi_Z$ for $Z\in\CO_C$, where \[
\psi_Z(X)=\tr(XZ) \ \  \textup{for all} \ \ X\in\CO_{C,1}.\]
Then it suffices to check that $\psi\ad_T=-\ad_{T,1}'\psi.$

Fix $Z\in\CO_C.$ Denote $l_1=[\psi\ad_T](Z)$ and $l_2=[\ad_{T,1}'\psi](Z)$. Then $l_i\in(\CO_{C,1})', i=1,2.$ It suffices to prove that $l_1=-l_2.$ Since
\[[\psi\ad_T](Z)=\psi[\ad_T(Z)]=\psi_{[T,Z]}, \ \ \ [\ad_{T,1}'\psi](Z)=\ad_{T,1}'(\psi_Z),\]
for any $X\in\CO_{C,1}$, we have
\begin{align*}
\l_1(X)&=\psi_{[T,Z]}(X)=\tr([T,Z]X)\\
&=\tr(TZX-ZTX)=-\tr(TXZ-XTZ)\\
&=-\tr([T,X]Z)=-\psi_Z([T,X])\\
&=-\psi_Z[\ad_{T,1}(X)]=-l_2(X).
\end{align*}
This shows that $l_1=-l_2.$
\end{proof}

The following corollary follows from Proposition \ref{L:Trace} and the closed range theorem.

\begin{corollary}\label{C:dual}
Let $T\in\CO_C$ and $1<p,q<\infty$ with $\frac{1}{p}+\frac{1}{q}=1.$ Then
\begin{enumerate}
 \item[(i)] $\sigma_\pi(\ad_{T})=-\sigma_\delta(\ad_{T,1})$ and $\sigma_\delta(\ad_{T})=-\sigma_\pi(\ad_{T,1});$
 \item[(ii)] $\sigma_\pi(\ad_{T,\infty})=-\sigma_\delta(\ad_{T,1})$ and $\sigma_\delta(\ad_{T,\infty})=-\sigma_\pi(\ad_{T,1});$
 \item[(iii)] $\sigma_\pi(\ad_{T,p})=-\sigma_\delta(\ad_{T,q})$ and $\sigma_\delta(\ad_{T,p})=-\sigma_\pi(\ad_{T,q}).$
\end{enumerate}
\end{corollary}

Now the proof of Theorem \ref{T:SpecDeriva} is an easy corollary.

\begin{proof}[Proof of Theorem \ref{T:SpecDeriva}]
By Corollary \ref{C:dual}, we have $\sigma_\delta(\ad_T)=-\sigma_\pi(\ad_{T,1})$.
It follows from Theorem \ref{T:app} that
\[\sigma_\delta(\ad_{T})=-[\sigma(T)+\sigma(T)]\setminus\Xi(2T).
\] In view of Lemma \ref{L:SSOBasic}, we obtain $\sigma_\delta(\ad_{T})=[\sigma(T)+\sigma(T)]\setminus\Xi(2T)$.
Furthermore,
\begin{align}\label{E-appp}
\sigma(\ad_{T})=\sigma_\pi(\ad_{T})=\sigma_\delta(\ad_{T})=[\sigma(T)+\sigma(T)]\setminus\Xi(2T).
\end{align}

It is well known that
$$\sigma_\pi(\ad_{T})\subset \sigma_l(\ad_{T})\subset \sigma(\ad_{T})\ \ \  \textup{and}
\ \ \ \sigma_\delta(\ad_{T})\subset \sigma_r(\ad_{T})\subset \sigma(\ad_{T}).$$ Thus the desired result follows readily.
\end{proof}

The proof of Theorem \ref{T:IdealSpecDeriva} follows the same line as that of Theorem \ref{T:SpecDeriva}.

\begin{proof}[Proof of Theorem \ref{T:IdealSpecDeriva}]
We just give the proof in the case that  $1<p<\infty$. The proof for $p\in\{1,\infty\}$ is similar.

Assume that $1<q<\infty$ with $\frac{1}{p}+\frac{1}{q}=1$.
By Theorem \ref{T:app}, we have
\begin{align*}
\sigma_\pi(\ad_{T,p})=[\sigma(T)+\sigma(T)]\setminus\Xi(2T)=\sigma_\pi(\ad_{T,q}).
\end{align*}
By Lemma \ref{L:SSOBasic} and Corollary \ref{C:dual},
the latter implies $\sigma_\delta(\ad_{T,p})=[\sigma(T)+\sigma(T)]\setminus\Xi(2T).$
Thus
\begin{align*}
\sigma(\ad_{T,p})=\sigma_\pi(\ad_{T,p})=\sigma_\delta(\ad_{T,p})=[\sigma(T)+\sigma(T)]\setminus\Xi(2T).
\end{align*}

It is well known that
$$\sigma_\pi(\ad_{T,p})\subset \sigma_l(\ad_{T,p})\subset \sigma(\ad_{T,p})\ \ \  \textup{and}
\ \ \ \sigma_\delta(\ad_{T,p})\subset \sigma_r(\ad_{T,p})\subset \sigma(\ad_{T,p}).$$ Thus the desired result follows readily.
\end{proof}

\begin{proof}[Proof of Theorem \ref{T:DerivFinDim}]
We view $T$ as a linear operator on $\bC^n$ (endowed with the usual inner product).
It suffices to prove that a point $z$ lies in $\Xi(T)$ if and only if $\textup{rank}(T-z)^2=n-1$ and there exist no distinct
$z_1,z_2\in \sigma(T)$ such that $2z=z_1+z_2$.

Since $T$ acts on an $n$-dimensional space, it follows that $\sigma(T)=\partial\sigma(T)=\iso\sigma(T)$.
Thus $\sigma(T)+\sigma(T)=\iso[\sigma(T)+\sigma(T)]$. Also $\textup{rank}(T-z)^2=n-1$ if and only if $\ker(T-z)^2=1$ or, equivalently,
$\dim\cH(z;T) =1$.
\end{proof}

\begin{remark}\label{R:comparison}
Let $T\in\CO_C$. By Lemma \ref{L:SSOBasic} and Corollary \ref{C:sum},
\begin{align*}
  \sigma_\pi(T)-\sigma_\delta(T)&=\sigma_\pi(T)+\sigma_\pi(T)=\sigma(T)+\sigma(T)\\
  &= -[\sigma(T)+\sigma(T)]=-[\sigma_\pi(T)+\sigma_\pi(T)]=\sigma_\delta(T)-\sigma_\pi(T).
\end{align*}
In view of \cite[Thm. 3.19]{Herr89}, we have
\begin{align*}
\sigma(\delta_T)=\sigma_l(\delta_T)=\sigma_\pi(\delta_T)=\sigma_r(\delta_T)=\sigma_\delta(\delta_T)=\sigma(T)+\sigma(T).
\end{align*}
Moreover, in view of \cite[Thm. 3.53 \& 3.54]{Herr89} or \cite{Fialkow}, we have
\begin{align*}
\sigma(\delta_T|_{\mathcal{J}})=\sigma_l(\delta_T|_{\mathcal{J}})=\sigma_\pi(\delta_T|_{\mathcal{J}})=\sigma_r(\delta_T|_{\mathcal{J}})=\sigma_\delta(\delta_T|_{\mathcal{J}})=\sigma(T)+\sigma(T),
\end{align*} where $\mathcal{J}=\B_p(\cH)$ for $p\in[1,\infty]$.
Thus, like $\ad_T$, the left spectra, the right spectra, the approximate point spectra and the approximate defect spectra of $\delta_T$ and $\delta_T|_{\mathcal{J}}$ all coincide.
\end{remark}

We conclude this paper with two examples.
First we give an example of $T\in\CO_C$ with $\Xi(T)\ne \emptyset$.

\begin{example}\label{E:sharp1}
Let $C$ be a conjugation on $\cH$ and $\{e_n\}_{n\geq 1}$ be an orthonormal basis of $\cH$ with $Ce_n=e_n$ for all $n$.
Define $T=\textrm{i}(e_1\otimes e_2-e_2\otimes e_1)$.
It is easy to check that $\sigma(T)=\{0,1,-1\}$ and $ \Xi(T)=\{-1,1\}$.
Then, by Theorem \ref{T:SpecDeriva}, $\sigma(\ad_T)=\{0\}$.
\end{example}

Here is an example of $T\in\CO_C$ with $\iso\sigma(T)\ne \emptyset=\Xi(T)$.

\begin{example}\label{E:sharp2}
Let $C$ be a conjugation on $\cH$ and $\{e_n\}_{n\geq 1}$ be an orthonormal basis of $\cH$ with $Ce_n=e_n$ for all $n$.
Define $T=\textrm{i}(e_1\otimes e_2-e_2\otimes e_1)+2\textrm{i}(e_3\otimes e_4-e_4\otimes e_3)+2\textrm{i}(e_5\otimes e_6-e_6\otimes e_5)$.
It is easy to check that $\sigma(T)=\{0,\pm1,\pm2\}$ and $ \Xi(T)=\emptyset$. Then, by Theorem \ref{T:SpecDeriva}, $$\sigma(\ad_T)=\sigma(T)+\sigma(T)=\{0,\pm1,\pm2,\pm 3,\pm 4\}.$$
\end{example}

 %%%%%%%%%%%%%

\section*{Acknowledgements}

This work was supported by the National Science Foundation of China (grant number 11671167).

%%%%%%%%%%%%%%%%%%%%%%%%%%%%%%%%%%%%%%%%%%%%

\end{document}